\documentclass[12pt]{amsart}
\usepackage{import}
\usepackage{comment}
\usepackage{mathrsfs}
\usepackage[margin=1in]{geometry}
\usepackage{subcaption}
\usepackage{hyperref}
\hypersetup{
	colorlinks=true,
	linkcolor=blue,
	filecolor=blue,      
	urlcolor=blue,
	citecolor=blue,
	pdftitle={Overleaf Example},
	pdfpagemode=FullScreen,
}

\usepackage{amsmath}
\usepackage{amsfonts}
\usepackage{amssymb}
\usepackage{graphicx}
\usepackage{tikz-cd}
\usepackage{amsthm}
\usepackage{adjustbox}

\newtheorem{counter}{counter}[section]
\newtheorem{theorem}[counter]{Theorem}

\newtheorem{lemma}[counter]{Lemma}
\newtheorem{proposition}[counter]{Proposition}

\theoremstyle{definition}
\newtheorem{definition}[counter]{Definition}
\newtheorem{remark}[counter]{Remark}

\newtheorem{example}[counter]{Example}
\newtheorem{conjecture}[counter]{Conjecture}
\newtheorem*{notation}{Notation}

\newcommand{\QQ}{\mathbb{Q}}

\newcommand{\PP}{\mathbb{P}}
\newcommand{\Z}{\mathbb{Z}}

\newcommand{\OO}{\mathcal{O}}

\newcommand{\ev}{\text{ev}}
\newcommand{\coev}{\text{coev}}

\newcommand{\Hom}{\text{Hom}}

\newcommand{\Ext}{\text{Ext}}
\newcommand{\ext}{\text{ext}}

\newcommand{\ch}{\text{ch}}

\newcommand{\perpp}{\text{perp}}

\title{Constructive exceptional bundles on $\PP^3$}
\author{Benjamin Gould}
\begin{document}
\maketitle
\begin{abstract}
We give a complete classification of the Chern characters of constructive exceptional vector bundles on $\mathbb{P}^3$ analogous to the work of Dr\'ezet and Le Potier on $\PP^2$, and using this classification prove that a constructive exceptional bundle $E$ on $\PP^3$ with $\mu(E) \geq 0$ is globally generated.
\end{abstract}
\tableofcontents

\section{Introduction} 
In this paper, we give a complete classification of the Chern characters of constructive exceptional bundles on $\PP^3$. As an application of the classification, we show that when $E$ is a constructive exceptional bundle on $\PP^3$ with $\mu(E) \geq 0$, then $E$ is globally generated. These results directly generalize analogous results obtained on $\PP^2$ respectively by Dr\'ezet and Le Potier \cite{DLP85} and Coskun-Huizenga-Kopper \cite{CHK}.

A sheaf $E$ is called exceptional if $\Hom(E,E) = \mathbb{C}$, and $\Ext^i(E,E) = 0$ for $i > 0$. Exceptional sheaves have proved to be significant and complex in lower dimensions and on other varieties. On $\PP^2$, exceptional sheaves are stable vector bundles which have a fractal structure, and they govern the classification of stable bundles \cite{DLP85}. On other Fano varieties, exceptional sheaves span the most easily accessible part of the derived category of coherent sheaves (see, e.g., \cite{Kuz}). As such, the structure of the exceptional bundles on a variety determines a significant amount of its geometry and homological algebra. 

On projective spaces, the simplest examples of exceptional bundles are the line bundles $\mathcal{O}_{\mathbb{P}^n}(d)$, and there is a process for forming new exceptional bundles from others, called mutation. We call an exceptional bundle on $\mathbb{P}^n$ \textit{constructive} if it is obtained from line bundles through mutation. Every exceptional bundle on $\PP^2$ is constructive \cite{Dr1}. In higher dimensional projective spaces there are no known examples of non-constructive exceptional sheaves. In this paper, we focus on constructive exceptional bundles on $\mathbb{P}^3$.

On $\mathbb{P}^3$, exceptional bundles are slope-stable \cite{Zube}. By definition, they are rigid, so their moduli spaces are zero-dimensional and reduced. Moduli spaces of stable sheaves on $\PP^3$ are central objects of study in algebraic geometry. Despite the fact that they underlie many of the richest and most well-studied problems appearing in the literature and provide some of its most important examples, they are often ungovernably complex \cite{Vak} and difficult to study. Constructive exceptional bundles provide interesting examples of stable bundles and moduli spaces on $\mathbb{P}^3$. In this paper, we classify these moduli spaces using the structure of mutation on $\mathbb{P}^3$ and prove when the associated bundles are globally generated.

\subsection{Exceptional bundles on $\PP^3$} Our first main result is the classification of Chern characters of constructive exceptional bundles on $\PP^3$, along the lines of the classification of exceptional slopes on $\PP^2$ by Dr\'ezet and Le Potier \cite{DLP85}.

Constructive exceptional bundles on $\PP^n$ appear in infinite ordered collections called \textit{helices}, any $n+1$ consecutive elements of which form a full exceptional collection. The simplest example is the standard helix $\sigma_0 = (\OO(i))_{i \in \Z}$ of line bundles. Other constructive helices are obtained from $\sigma_0$ via a process called mutation.

\begin{theorem}[= Theorem \ref{theorem}] \label{main}
Let $E$ be a constructive exceptional bundle on $\PP^3$. There is a well-defined way to choose a distinguished constructive helix $\sigma = \sigma(E)$ containing $E$ and a full exceptional collection $(E_1, E, E_2, E_3)$ of sheaves in $\sigma$ containing $E$. For such a collection, we write $\ch(E) = perp(E_1, E_2, E_3)$. This association produces a bijection
	\[\epsilon_{\PP^3}: \{\text{3-adic rationals}\} \rightarrow \left\{\begin{array}{c} \text{Chern characters of constructive} \\ \text{exceptional bundles on } \PP^3 \end{array} \right\}\]
defined inductively as
	\[\epsilon_{\PP^3}(n) = n \quad n \in \Z,\]
and for $q \geq 0$,
	\begin{align*}
		\epsilon_{\PP^3}\left( \frac{3p+1}{3^{q+1}} \right) &= \text{perp}\left( \epsilon_{\PP^3}\left( \frac{p}{3^q} \right), E_2\left(\epsilon_{\PP^3}\left( \frac{p}{3^q} \right)\right), E_3\left( \epsilon_{\PP^3}\left( \frac{p}{3^q} \right)\right) \right) \\
		\epsilon_{\PP^3}\left( \frac{3p+2}{3^{q+1}} \right) &= \text{perp}\left( E_1\left( \epsilon_{\PP^3}\left( \frac{p+1}{3^q} \right) \right), \epsilon_{\PP^3}\left( \frac{p+1}{3^q} \right), E_3\left( \epsilon_{\PP^3}\left( \frac{p+1}{3^q} \right) \right) \right).
	\end{align*}
\end{theorem}

The proof of Theorem \ref{main} involves a close examination of the helix theory on $\PP^3$, which we organize into a graph $\text{H}_{\PP^3}$ whose vertices are the constructive helices, and whose edges correspond to mutation. The graph $\text{H}_{\PP^3}$ carries a distinguished subgraph $\Gamma_{\PP^3}$, which is a tree, and every constructive exceptional bundle appears in a helix in $\Gamma_{\PP^3}$ (Proposition \ref{gamma-tree}). The proof of Theorem \ref{main} is an identification of this tree with the numerical data of exceptional bundles appearing in these helices.

Once the structure of the mutations forming constructive exceptional bundles is described, their study is much more accessible. Our primary application is the following theorem.

\begin{theorem}[= Theorem \ref{gg}]
Let $E$ be a constructive exceptional bundle on $\PP^3$ with $\mu(E) \geq 0$. Then $E$ is globally generated.
\end{theorem}

The paper is organized as follows. Section \ref{prelim-section} reviews the theory of exceptional bundles and helices on projective space. In Section \ref{graphs-section} we analyze the structure of mutations of helices on $\PP^3$ and obtain the necessary results for the classification, which we carry out in Section \ref{class-section}. In Section \ref{explicit-section} we include explicit computations of the low-order exceptional bundles on $\PP^3$, and in Section \ref{gg-section} we prove global generation.

\subsection{Acknowledgements} I am happy to thank Izzet Coskun, Jack Huizenga, and Yeqin Liu for indispensible conversations. I am happy to thank Boris Pioline for pointing out errors in Table \ref{fig:table} in a previous version of this paper.

\section{Preliminaries} \label{prelim-section}
\subsection{Exceptional bundles and helix theory on $\PP^n$} Exceptional bundles were first introduced by Dr\'ezet and Le Potier in \cite{DLP85} in order to classify the Chern characters of stable bundles on $\PP^2$ (see Section \ref{Ptwo} for further details). The theory was expanded to any projective space $\PP^n$ by Gorodentsev-Rudakov \cite{GR} and studied much more generally by the Rudakov seminar \cite{Rud}. Exceptional bundles ordinarily fit together in exceptional collections, whose transformations are known as mutations, which naturally occur in helices. We define all of these notions now. Note that the theory works just as well on any Fano variety, but we only form our definitions on $\PP^n$.

Throughout, we will let $\hom(A,B) = \dim \Hom(A,B)$ and $\ext^i(A,B) = \dim \Ext^i(A,B)$ for sheaves $A, B$. 

\begin{definition}
A sheaf $E$ on $\PP^n$ is \textit{exceptional} if $\hom(E,E) = 1$ and $\ext^i(E,E) = 0$ for $i > 0$. An ordered collection $(E_1, ..., E_k)$ of sheaves is an \textit{exceptional collection} if each $E_i$ is exceptional and in addition we have $\Ext^m(E_i, E_j) = 0$ for $m \geq 0$ when $i > j$. When the extension-closure of an exceptional collection generates the derived category $D^b(\PP^n)$, we say it is \textit{full}. A full exceptional collection on $\PP^n$ has length $n+1$. 
\end{definition}

The simplest examples of exceptional sheaves on $\PP^n$ are the line bundles $\OO_{\PP^n}(d)$. Every exceptional sheaf on $\PP^n$ is a vector bundle \cite{Pos}. One can also form these definitions for objects in the derived category; all exceptional objects on $\PP^n$ are shifts of vector bundles.

Given an ordered pair of sheaves $(E, F)$, we form the evaluation and coevaluation maps
	\[\ev: E \otimes \Hom(E,F) \rightarrow F, \quad \coev: E \rightarrow F \otimes \Hom(E,F)^{\vee}\]
each of which is associated to the identity element of the space $\Hom(E,F) \otimes \Hom(E,F)^{\vee}$. If the evaluation map is surjective it determines an exact sequence
	\[0 \rightarrow L_E F \rightarrow E \otimes \Hom(E,F) \stackrel{\ev}{\rightarrow} F \rightarrow 0\]
and if the coevaluation map is injective it determines an exact sequence
	\[0 \rightarrow E \stackrel{\coev}{\rightarrow} F \otimes \Hom(E,F)^{\vee} \rightarrow R_F E \rightarrow 0.\]
	
\begin{definition}
The sheaf $L_E F$ is the \textit{left mutation} of $F$ across $E$, and the sheaf $R_F E$ is the \textit{right mutation} of $E$ across $F$. We denote mutations by arrows $(E,F) \mapsto (F, R_FE), (L_EF, E)$.
\end{definition} 

If $(E,F)$ is an ordered pair of exceptional bundles, whenever the left and right mutation sheaves $L_E F$ and $R_F E$ are defined, they are exceptional as well \cite[Proposition 1.5]{GR}. Thus mutation provides a way to form new exceptional bundles from old ones.

\begin{example} \label{euler}
	The Euler sequence
	\[0 \rightarrow \OO_{\PP^n} \rightarrow \OO_{\PP^n}(1) \otimes \Hom(\OO_{\PP^n}, \OO_{\PP^n}(1))^{\vee} \rightarrow T_{\PP^n} \rightarrow 0\]
	realizes $T_{\PP^n}$ as the right mutation $R_{\OO_{\PP^n}(1)}\OO_{\PP^n}$. It follows that $T_{\PP^n}$ is also exceptional.
\end{example}

We call an infinite ordered sequence of sheaves $(E_i)_{i \in \mathbb{Z}}$ on $\mathbb{P}^n$ \textit{periodic} of period $m$ if for each $i$ and $k$, we have $E_{i - mk} \simeq E_i \otimes \omega_{\mathbb{P}^n}^{\otimes k} \simeq E_i(k(-n-1))$. If $m = n+1$, we will simply call the collection periodic.

From a full exceptional collection $(E_1, ..., E_{n+1})$ on $\PP^n$ one can form an infinite periodic collection $(E_i)_{i \in \Z}$ by setting $E_{i - (n+1)k} = E_i \otimes \omega^k_{\PP^n} = E_i(k(-n-1))$ for $0 \leq i \leq n+1$ and $k \in \Z$. Clearly such a collection on $\PP^n$ is determined by any $n+1$ consecutive sheaves, which we call a \textit{foundation}. 

Mutations can also be formed in these infinite periodic collections: if $(E_i, E_{i+1})$ has a surjective evaluation map (resp. injective coevaluation map), then one chooses a foundation $(E_i, E_{i+1}, ..., E_{i+n+1})$ including $(E_i, E_{i+1})$; forms the mutation pair $(L_{E_i}E_{i+1}, E_i)$; mutates the foundation to obtain a new collection $(L_{E_i}E_{i+1}, E_i, ..., E_{i+n+1})$ and then extends this collection to an infinite periodic collection. Entirely similar statements hold for right mutations.

When the operations are defined, one can iterate mutation. We write $L_j(E_i)_{i \in \Z}$ for the left mutation $L_{E_j}$ (so $L_j$ shifts/replaces the $j$th bundle in the sequence). We include the following example for clarity.

\begin{example} 
	Let $\sigma = (\mathcal{O}_{\mathbb{P}^3}(-1), \mathcal{O}_{\mathbb{P}^3}, \mathcal{O}_{\mathbb{P}^3}(1), \mathcal{O}_{\mathbb{P}^3}(2))$ be the standard exceptional collection on $\mathbb{P}^3$. We extend $\sigma$ to an infinite collection $(E_i) = (E_i)_{i \in \mathbb{Z}}$ of line bundles by setting $E_i = \mathcal{O}_{\mathbb{P}^3}(i)$; this collection is clearly periodic. 
	
	We form the mutation $L_0(E_i)$ to demonstrate the operations. The mutation $L_0(E_i)$ is the infinite collection obtained from $(E_i)$ by taking the left mutation of the pair $(E_0, E_1) = (\mathcal{O}_{\mathbb{P}^3}, \mathcal{O}_{\mathbb{P}^3}(1))$, and then extending this mutation to the rest of the collection. Observe that $L_{\mathcal{O}_{\mathbb{P}^3}}\mathcal{O}_{\mathbb{P}^3}(1) = T_{\mathbb{P}^3}^{\vee}(1)$, as follows from Example \ref{euler}. Then to perform the mutation on the collection $(E_i)$, we first mutate $\sigma$ as follows:
		\[(\mathcal{O}_{\mathbb{P}^3}(-1), \mathcal{O}_{\mathbb{P}^3}, \mathcal{O}_{\mathbb{P}^3}(1), \mathcal{O}_{\mathbb{P}^3}(2)) \mapsto (\mathcal{O}_{\mathbb{P}^3}(-1), T_{\mathbb{P}^3}^{\vee}(1), \mathcal{O}_{\mathbb{P}^3}, \mathcal{O}_{\mathbb{P}^3}(2)).\]
	Then we extend this exceptional collection to an infinite periodic collection $(F_i) = (F_i)_{i \in \mathbb{Z}}$. We set $F_0 = \mathcal{O}_{\mathbb{P}^3}(-1)$, $F_1 = T_{\mathbb{P}^3}^{\vee}(1)$, $F_2 = \mathcal{O}_{\mathbb{P}^3}$ and $F_3 = \mathcal{O}_{\mathbb{P}^3}(2)$. Then we extend to all integers by setting $F_{i-4k} = F_i(-4k)$ for each $k \in \mathbb{Z}$. E.g., we have $F_5 = F_{1+4} = T_{\mathbb{P}^3}^{\vee}(5)$.
\end{example}

\begin{definition} \label{helix-defn}
An ordered infinite periodic collection $\sigma = (E_i)_{i \in \Z}$ of exceptional bundles on $\mathbb{P}^n$ is a \textit{helix} if
	\begin{enumerate}
		\item each evaluation map $E_i \otimes \Hom(E_i, E_{i+1}) \rightarrow E_{i+1}$ is surjective,
		\item each coevaluation map $E_i \rightarrow E_{i+1} \otimes \Hom(E_i, E_{i+1})^{\vee}$ is injective, and
		\item We have $(L_{j-n} \cdots L_{j-1}L_j)\sigma = \sigma$ and $(R_{j+n} \cdots R_{j+1}R_j)\sigma = \sigma$.
	\end{enumerate}
The last condition is known as the ``helix condition.'' It is equivalent to the following condition, which will be more useful for our applications:
	\begin{enumerate}
		\item[($3^{\prime}$)] $(L_{j-(n-k)} \cdots L_{j-1}L_j)\sigma = (R_{j-(n-k)-1} \cdots R_{j-n})\sigma$ for $0 \leq k \leq n$.
	\end{enumerate}
\end{definition}

The most important example of a helix is the collection of line bundles $\sigma_0 = (\OO(i))_{i \in \Z}$; for a proof that this is indeed a helix, see \cite[Proposition 1.11]{GR}. The most important property of helices is that they are preserved under mutation: a mutation of a helix is a helix \cite[Theorem 2.1]{GR}.

\begin{definition}
A helix $\sigma$ (or bundle $E$) on $\PP^n$ is called \textit{constructive} if it can be obtained from the standard helix $\sigma_0$ by a sequence of mutations. 
\end{definition}

Beilinson showed that line bundles form \textit{full} exceptional collections, i.e., that they generate the derived category \cite{Bei}, and Bondal-Gorodentsev generalized this result to constructive helices \cite{BG}. There are no known examples of helices or exceptional bundles on any $\PP^n$ which are not constructive. 

\subsection{Exceptional bundles on $\PP^2$} \label{Ptwo} Exceptional bundles on $\PP^2$ were intensively studied by Dr\'ezet and Le Potier \cite{DLP85,Dr1,LP}. Every exceptional bundle on $\PP^2$ is constructive and slope-stable \cite{Dr1,LP}. Their moduli spaces consist of a single reduced point \cite{LP}. 

Vector bundles on $\PP^2$ are often described by their rank $r = \ch_0$, slope $\mu = \frac{\ch_1}{\ch_0}$, and discriminant $\Delta = \frac{1}{2}\mu^2 - \frac{\ch_2}{\ch_0}$. The Riemann-Roch theorem implies that $\Delta = \frac{1}{2}(1 - \frac{1}{r^2})$. The rank and degree of an exceptional bundle on $\PP^2$ are coprime, and the discriminant is a function of the rank, so the slope determines the rank and discriminant, hence the entire Chern character.

The slopes of exceptional bundles have been classified, by the following theorem of Dr\'ezet and Le Potier. (See also \cite[16.3]{LP}.) Set $\mathcal{E}$ to be the set of slopes of exceptional bundles on $\PP^2$. For any rational number $\alpha$ we define $r(\alpha)$ to be the smallest integer such that $r(\alpha)\alpha \in \Z$, and
	\[\Delta(\alpha) =  \frac{1}{2}\left( 1 - \frac{1}{r(\alpha)^2} \right), \quad \chi(\alpha) = r(\alpha)(P(\alpha) - \Delta(\alpha)),\]
where $P(x) = \frac{1}{2}x^2 + \frac{3}{2}x + 1$ is the Hilbert polynomial of $\mathcal{O}_{\mathbb{P}^2}$.	

The Riemann-Roch theorem implies that when $E$ is exceptional on $\PP^2$ with $\mu(E) = \alpha$, we have $\Delta(E) = \Delta(\alpha)$ and $\chi(E) = \chi(\alpha)$. We say a pair of rational numbers $\alpha, \beta$ is \textit{admissible} if are slope-close, cohomologically orthogonal, and have integral Euler characteristics. (See \cite[Definition 16.3.2]{LP} for a precise definition.)

For a pair of admissible rational numbers, one then introduces the ``dot'' operator:
	\[\alpha.\beta = \frac{\alpha + \beta}{2} + \frac{\Delta(\beta) - \Delta(\alpha)}{3 + \alpha - \beta}.\]

From this set-up, Dr\'ezet and Le Potier define a map $\epsilon_{\PP^2}: \{\text{dyadic rationals}\} \rightarrow \mathcal{E}$ inductively as follows: $\epsilon_{\PP^2}(n) = n$, and for $q \geq 0$,
	\[\epsilon_{\PP^2}\left( \frac{2p+1}{2^{q+1}} \right) = \epsilon_{\PP^2}\left( \frac{p}{2^q} \right).\epsilon_{\PP^2}\left( \frac{p+1}{2^q} \right).\]

\begin{theorem}\cite[Theorem 16.3.4]{LP}
The map $\epsilon_{\PP^2}$ is a bijection.
\end{theorem}

The conditions of admissibility encode the cohomological orthogonality of exceptional collections as in the previous section. The third condition is vacuous when $\alpha, \beta$ are the slopes of exceptional bundles. The second specifies that $\chi(E_{\beta}, E_{\alpha}) = 0$, where $E_{\mu}$ is the exceptional bundle of slope $\mu$. And one can readily check that $\epsilon(\alpha.\beta) = \mu(R_{E_{\gamma}}E_{\alpha})$ where $(E_{\gamma},E_{\alpha}, E_{\beta})$ is a full exceptional collection. Note that on $\PP^2$, the helix condition reads $R_{E_{\gamma}}E_{\alpha} = L_{E_{\alpha}}E_{\beta}$; left mutations are also right mutations. 

In Section \ref{class-section}, we will see how these properties generalize to $\PP^3$.

\subsection{Exceptional bundles and stability on $\PP^2$} Exceptional bundles on $\PP^2$ are significant in the first place because of their appearance in the classification of the Chern characters of stable bundles, due to Dr\'ezet and Le Potier. Since exceptional bundles on $\PP^2$ are stable, given an exceptional bundle $E_{\alpha}$ and a stable bundle $E$ of slope $\mu$ with $0 \leq \mu - \alpha < 3$, one sees from stability and Serre duality that
	\[\hom(E, E_{\alpha}) = 0 \quad \text{and} \quad  \ext^2(E, E_{\alpha}) = \hom(E_{\alpha}, E(-3)) = 0\]
so by Riemann-Roch
	\[P(\alpha - \mu) - \Delta_{\alpha} < \Delta(E).\]
Dr\'ezet and Le Potier's classification shows that this numerical constraint on $\Delta(E)$ is the only constraint for stability.

\begin{theorem}\cite{DLP85} \label{P2class}
	Set 
		\[\delta(\mu) = \sup_{\alpha \in \mathcal{E}, |\alpha-\mu|<3}\{P(-|\alpha - \mu|) - \Delta(\alpha)\}.\]
	Then $(r, \mu, \Delta)$ is the Chern character of a stable bundle if and only if it is exceptional, or its Euler characteristic and Chern classes are integers with
		\[\delta(\mu) \leq \Delta.\]
\end{theorem}

\subsection{Exceptional bundles on $\PP^3$} Exceptional bundles on $\PP^3$ are not as well-understood as those on $\PP^2$, however, some important results have been obtained. Zube showed that each exceptional bundle on $\PP^3$ is slope-stable by restricting to an anticanonical K3 surface, by showing the restriction to the K3 surface is spherical and slope-stable \cite{Zube}. Nogin classified the values of full exceptional collections in the Grothendieck group (i.e. Chern characters) and showed that each is equal to the value in the Grothendieck group of a constructive exceptional collection \cite{Nog}. 

Polishchuk built on these results. It follows from the vanishing result \cite[Theorem 1.2]{Pol} that any exceptional bundle on $\mathbb{P}^3$ with the Chern character equal to that of a constructive exceptional bundle is itself constructive. It follows that the moduli space of a constructive exceptional bundle is a single reduced point. In fact it is proved that the action of the braid group on the set of \textit{simple} helices on some Fano threefolds, including $\mathbb{P}^3$, is transitive \cite[Theorem 1.1]{Pol}, using Nogin's results. This does not quite accomplish the classification of exceptional bundles on $\PP^3$, as it is not clear that an exceptional bundle can be included into a full exceptional collection.

Each full exceptional collection on $\PP^3$ has length 4, and we often work with helices via the choice of a foundation, which forms a full exceptional collection. Let $\sigma$ denote a constructive helix on $\PP^3$, and choose a foundation $(E,F,G,H)$ for $\sigma$. 

\begin{lemma} \label{mutations}
\begin{enumerate}
	\item The slopes of the elements of a foundation are ordered: 
		\[\mu(E) < \mu(F) < \mu(G) < \mu(H).\]
	\item The helix $\sigma$ has exactly 8 distinct mutations, which are given on this foundation by the mutations of the pairs
		\[(E, F), \quad (F, G), \quad (G,H),\]
	the right mutation of the pair $(H(-4), E)$, and the left mutation of the pair $(H,E(4))$.
\end{enumerate}
\end{lemma}

\begin{proof}
For each pair $(A,B)$ of bundles appearing consecutively in a helix we can form the evaluation map $\ev: A \otimes \Hom(A,B) \rightarrow B$, which is surjective. Since exceptional bundles on $\PP^3$ are slope-stable \cite{Zube}, it follows that $\mu(A) = \mu(A \otimes \Hom(A,B)) < \text{image}(\ev) < \mu(B)$. Item (1) follows.

There are 8 possible mutations of a given helix on $\PP^3$, defined on a foundation by the statement of the Lemma; we need to show that each is distinct. Applying each possible mutation to the foundation $(E,F,G,H)$ and applying (1), we get inequalities
	\begin{align*}
		\mu(E) &< \mu(R_E(H(-4))), \mu(L_FG) < \mu(F), \\
		\mu(F) &< \mu(R_FE), \mu(L_GH) < \mu(G), \\ 
		\mu(G) &< \mu(R_GF), \mu(L_HF(4)) < \mu(H).
	\end{align*}
It remains to show that the mutations whose slopes lie between the slopes of bundles in the original foundation are similarly ordered. These mutations all commute with one another, so we can form double mutation foundations, e.g., $(F,R_FE,L_GH,G)$. The remaining inequalities follow again from Item (1).
\end{proof}



\subsection{The helix condition on $\PP^3$} \label{helix} We will repeatedly exploit the helix condition to produce new helices containing a given bundle. If $\sigma$ is a helix on $\PP^3$ with foundation $(E,F,G,H)$, we usually apply the helix condition in the forms
	\[R(E,F) = L(G,L_H(E(4))) \circ L(H, E(4)) \quad \text{and} \quad L(G,H) = R(R_{E(4)}H, F(4)) \circ R(H,E(4)).\]
On foundations, these are:

\adjustbox{scale = .8, center}{
	\begin{tikzcd}
		{(E,F,G,H)} \arrow[d] \arrow[r]                & {(F,G,L_H(E(4)),H)} \arrow[d] &  & {(E,F,G,H)} \arrow[d] \arrow[r]               & {(E,R_E(H(-4)),F,G)} \arrow[d] \\
		{(F,R_FE, G,H)} \arrow[r, equals] & {(F,L_G(L_H(E(4))),G,H)}                   &  & {(E,F,L_GH,G)} \arrow[r, equals] & {(E,F,R_F(R_E(H(-4))),G)}                  
	\end{tikzcd}
}

When we have an instance of the helix relation of the form $\gamma_1 = \gamma_2 \circ \gamma_3$, we will often refer to these as a \textit{helix triangle}. See Figure \ref{fig:some-adm-muts-2} for an illustration.

The braid group naturally acts on the set of full exceptional collections via mutations. This action gives a way to identify new sets of mutations which produce the same helix, but we will not use it in the sequel. Instead, we will use the helix relation to produce new mutations.

\subsection{Exceptional bundles and stability on $\PP^3$} There are several natural ways to try to extend the classification of stable Chern characters on $\PP^2$ in terms of exceptional bundles to $\PP^3$. On $\PP^2$, when $F$ is a stable bundle of slope $\mu$ and $E_{\alpha}$ is an exceptional bundle of slope $\alpha$ with $0 < \alpha - \mu < 3$, then by stability $\Hom(F, E_{\alpha}) = 0$ and by Serre duality $\Ext^2(F, E_{\alpha}) = 0$, so $\chi(F, E_{\alpha}) < 0$; applying this bound, and the entirely similar bound obtained from exceptional slopes $\alpha$ with $0 < \mu - \alpha < 3$, for each exceptional bundle produces precisely the condition that $\Delta(F) \geq \delta(\mu)$ from Theorem \ref{P2class}. 

If $F$ is a stable bundle on $\mathbb{P}^3$ of slope $\mu$ and $E_{\alpha}$ is a constructive exceptional bundle of slope $\alpha$ with $0 < \mu - \alpha < 4$, then, again by stability and Serre duality, 
	\[\hom(F, E_{\alpha}) = 0 \quad \text{and} \quad  \ext^3(F, E_{\alpha}) = \hom(E_{\alpha}, F(-4)) = 0.\]
If $\Ext^2(F, E_{\alpha}) = 0$, then we would deduce an inequality $\chi(F,E_{\alpha}) < 0.$ 

If similarly we have $0 < \alpha - \mu < 4$ and a vanishing $\Ext^2(E_{\alpha}, F) = 0$, then we would also conclude that $\chi(E_{\alpha}, F) < 0$, and these inequalityies together would produce a lower bound on $\Delta(F)$ analogous to that on $\PP^2$. Unfortunately, the following example, which applies recent work of Coskun-Huizenga-Smith \cite{CHS}, shows that this $\Ext^2$ vanishing property does not hold in general, so this procedure does not produce the classification of stable Chern characters on $\PP^3$.

\begin{example}
Consider a bundle $V$ on $\PP^3$ which is defined via an exact sequence
	\[0 \rightarrow \OO_{\mathbb{P}^3}(-4)^8 \stackrel{M}{\rightarrow} \OO_{\mathbb{P}^3}(-3)^{11} \rightarrow V \rightarrow 0,\]
where $M$ is a general matrix of linear forms. Then \cite[Theorem 5.1]{CHS} implies that $V$ is slope-stable. In particular, since $h^i(\OO_{\mathbb{P}^3}(-3)) = 0$ for all $i$ and $h^3(\OO_{\mathbb{P}^3}(-4)) = 1$ with $h^i(\OO_{\mathbb{P}^3}(-4)) = 0$ otherwise, we have $H^2(V) = \Ext^2(\OO_{\mathbb{P}^3}, V) = 8$, which contradicts the above vanishing statement. The discriminant $\Delta(V)$ does not lie above the lower bounds predicted above.
\end{example}

In other words, exceptional bundles on $\mathbb{P}^3$ do not determine the sharp Bogomolov-Gieseker inequality, unlike $\mathbb{P}^2$. The sharp Bogomolov-Gieseker inequality on $\PP^3$ remains an interesting and difficult area of research (see, e.g., \cite{Sch}).

\section{Graphs associated to helices on $\PP^3$} \label{graphs-section} It is convenient and natural to associate a graph $\text{H}_{\PP^3}$ to the set of constructive exceptional helices, whose vertices correspond to the helices and whose edges correspond to mutations. In fact we will focus on a sequence of subgraphs of $\text{H}_{\PP^3}$ which are easier to understand. Each graph will be defined inductively.

See Example \ref{explicit-muts} for a concrete example of the following abstract definitions. In this and the following sections, we will only be concerned about bundles on $\mathbb{P}^3$, so we will often abbreviate $\mathcal{O}_{\mathbb{P}^3}$ to $\mathcal{O}$ to unburden the notation.

\begin{definition}
	We define the graph $\text{H}_{\PP^3}$ as follows. In step 0, we add a vertex corresponding to the standard helix $\sigma_0$. In step 1, we add vertices for each of its 8 mutations, connected to $\sigma_0$ via an edge. For a vertex $v$ corresponding to a helix $\tau$ added in step $n$, in step $n+1$ we add a vertex for each mutation of $\tau$ which does not already appear in a previous step, and connect each of the mutations of $\tau$ to $\tau$ by an edge. We set $\text{H}_{\PP^3}$ to be the infinite graph constructed in this way.
	
	There are precisely 8 distinct mutations of any helix on $\PP^3$, so each vertex of $\text{H}_{\PP^3}$ has degree 8. Because mutations are invertible, we do not direct the edges of $\text{H}_{\PP^3}$.
\end{definition}

We will define a subgraph $\Gamma_{\PP^3} \subseteq \text{H}_{\PP^3}$ whose inductive definition is determined by helix theory on $\PP^3$. The graph $\text{H}_{\PP^3}$ is rather complicated, but we will show that the subgraph $\Gamma_{\PP^3}$ (and a subgraph $\Gamma_{\PP^3}' \subseteq \Gamma_{\PP^3}$) contain all of the information relevant for our purposes, and are much simpler. Both $\Gamma_{\mathbb{P}^3}$ and $\Gamma_{\mathbb{P}^3}'$ are trees (see Proposition \ref{gamma-tree}), and every constructive exceptional bundle appears in a helix in $\Gamma_{\PP^3}$ (see Proposition \ref{adm-replacement}).

We first establish some terminology. Fix a pair of helices and a mutation $\gamma: \sigma \mapsto \tau$ between them. In terms  of this mutation, the 8 mutations of $\tau$ can be further isolated. 

\begin{definition} \label{admissible}
	We partition the mutations described in Lemma \ref{mutations} in terms of the mutation $\gamma$ as follows. Suppose on a foundation $(E,F,G,H)$ for $\sigma$ the mutation $\gamma$ is defined on the pair $(E,F)$; the definitions differ for left and right mutations. We set:
	
	\begin{center}
		$\begin{array}{c}
			\gamma = R(E,F) \\
			\Rightarrow \tau = (F, R_F E, G, H):		
		\end{array} $
		\qquad 
		\begin{tabular}{c | c | c }
			$\gamma$-commuting & $\gamma$-admissible & $\gamma$-extraneous \\ \hline
			$R(G,H)$, & $R(F, R_F E)$,  & $L(F, R_F E)$,  \\
			$L(G,H)$ & $L(R_F E, G),$ & $R(G, R_F E)$, \\
			&  $R(H(-4), F)$ & $L(H, F(4))$
		\end{tabular}
	\end{center}
	and
	\begin{center}
		$\begin{array}{c}
			\gamma = L(E,F)  \\
			\Rightarrow \tau = (L_E F, E, G, H):		
		\end{array} $
		\begin{tabular}{c | c | c }
			$\gamma$-commuting & $\gamma$-admissible & $\gamma$-extraneous \\ \hline
			$R(G,H)$, & $L(L_E F, E)$ & $R(L_E F, E)$, \\
			$L(G,H)$ & $R(H(-4), L_F E)$, &  $L(H(-4), L_E F)$, \\
			& $L(E,G)$ & $R(E,G)$
		\end{tabular}
	\end{center}
If the mutation $\gamma$ is clear from the context, we will often shorten ``$\gamma$-admissible'' to simply ``admissible'', and we will often also speak of ``admissible helices'', which are those obtained from a fixed helix via mutations which are each admissible with respect to the previous one.
\end{definition}

\begin{remark} \label{commuting-remark}
	Definition \ref{admissible} is meant to distinguish a class of mutations, the $\gamma$-admissible mutations, which form bundles with slope close to the slope of the bundle produced by $\gamma$, when considered on a preferred foundation. In slogan form, the $\gamma$-commuting mutations produce new bundles inefficiently, and the $\gamma$-extraneous mutations either produce new bundles inefficiently or produce bundles with slope more than 4 from the bundle produced by $\gamma$; these bundles are more efficiently produced by admissible mutations in other slope ranges. See Proposition \ref{adm-replacement} for precise statements to this effect.
	
	Observe that if $\gamma: \sigma \mapsto \tau$ is a mutation, then $\delta: \tau \mapsto \tau'$ being $\gamma$-commuting means that there is a commutative square of mutations
	\begin{center}
		\begin{tikzcd}
			\tau \arrow[r, "\delta"] & \tau' \\ 
			\sigma \arrow[u, "\gamma"] \arrow[r] & \sigma' \arrow[u]
		\end{tikzcd}
	\end{center}
	since (in the notation of the definition) the pair $(G,H)$ appears in $\sigma$ and $\tau$.
	
	The $\gamma$-extraneous mutations can be described as follows. The first listed mutation in either case is the inverse of the mutation $\gamma$; the second listed mutation is the mutation in a helix triangle following $\gamma$; the third produces a bundle $L_H(F(4))$ or $R_GE$, whose slope differs by more than 4 from the slope of the new bundle introduced by $\gamma$.
\end{remark}

\begin{example} \label{adm-ex}
	We give descriptions of admissible mutations which will be important in later arguments (\textit{cf.} Proposition \ref{adm-replacement}). See Example \ref{explicit-muts} below for a concrete example. 
	
	Suppose that $\gamma: \sigma \mapsto \sigma_1$ is a mutation, and $\gamma':  \sigma_1 \mapsto \tau_1$ is $\gamma$-commuting. Then there is another $\gamma$-commuting mutation $\sigma_1 \mapsto \tau_1'$ of $\sigma_1$, which mutates the same pair in the other direction. We obtain squares as in Remark \ref{commuting-remark}. Then we can apply $\gamma$ to $\sigma_1$ as well, by mutating the new pair in the same way (e.g., $\gamma$ replaces $(E,F)$ by $(F, R_FE)$ and then replaces this by $(R_FE, R_{R_FE}F)$). See Figure \ref{fig:some-adm-muts-1}. Adding in the helix triangles that the mutations of $\sigma_1$ fit into, we obtain the additions in Figure \ref{fig:some-adm-muts-2}.
	
	This diagram shows all of the $\gamma$-admissible mutations. First, $\gamma$ applied to $\sigma_1$ is admissible, and is of the first kind in the definition of a $\gamma$-admissible mutation (in both the left and right mutation cases). The mutations $\sigma_1 \mapsto \epsilon_1, \epsilon_2$ are the other $\gamma$-admissible mutations. To see this, consider the following example (the others are similar): if $\gamma$ on foundations is the right mutation $(E,F,G,H) \mapsto (F, R_FE, G, H)$, and the commuting mutation $\gamma'$ does $(F, R_FE, G, H) \mapsto (F, R_FE, L_GH, G)$ on this foundation, then, in light of Section \ref{helix}, the mutation to $\epsilon_1$ is $(F, R_FE, G, H) \mapsto (F,R_F(H(-4)), R_FE, G)$, which is the third kind of $\gamma$-admissible mutation in the right mutation case in Definition \ref{admissible}.
	
	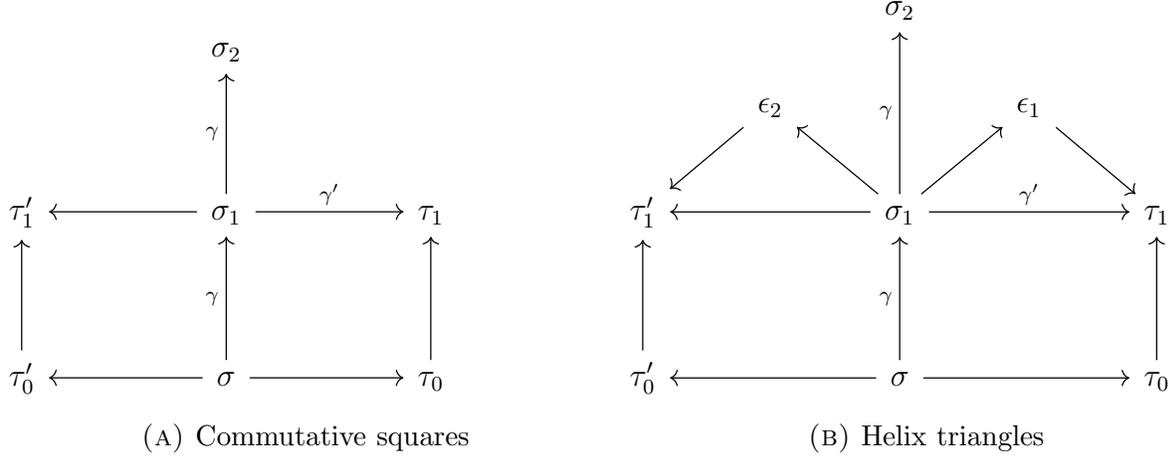
\begin{figure}
		\begin{subfigure}{.5\textwidth}
			\begin{tikzcd}
				&  & \sigma_2                                                       &  &                   \\
				&  &                                                                &  &                   \\
				\tau_1'            &  & \sigma_1 \arrow[rr, "\gamma'"] \arrow[uu, "\gamma"] \arrow[ll] &  & \tau_1            \\
				&  &                                                                &  &                   \\
				\tau_0' \arrow[uu] &  & \sigma \arrow[uu, "\gamma"] \arrow[rr] \arrow[ll]            &  & \tau_0 \arrow[uu]
			\end{tikzcd}
			\caption{Commutative squares}
			\label{fig:some-adm-muts-1}
		\end{subfigure}%
		\begin{subfigure}{.5\textwidth}
			\begin{tikzcd}
				&                       & \sigma_2                                                                             &                       &                   \\
				& \epsilon_2 \arrow[ld] &                                                                                      & \epsilon_1 \arrow[rd] &                   \\
				\tau_1'            &                       & \sigma_1 \arrow[rr, "\gamma'"] \arrow[uu, "\gamma"] \arrow[ll] \arrow[ru] \arrow[lu] &                       & \tau_1            \\
				&                       &                                                                                      &                       &                   \\
				\tau_0' \arrow[uu] &                       & \sigma \arrow[uu, "\gamma"] \arrow[rr] \arrow[ll]                                  &                       & \tau_0 \arrow[uu]
			\end{tikzcd}
			\caption{Helix triangles}
			\label{fig:some-adm-muts-2}
		\end{subfigure}
		\caption{Mutations for Example \ref{adm-ex}.}
		\label{fig:some-adm-muts}
	\end{figure}
\end{example}

\begin{example} \label{explicit-muts}
	Consider the mutation $\gamma: \sigma_0 = (\mathcal{O}(i))_{i \in \mathbb{Z}} \mapsto \sigma_1$ which produces the tangent bundle $T_{\mathbb{P}^3}$. Specifically, $\gamma$ is defined by the Euler sequence \ref{euler}, and $\sigma_1$ has foundation $(\mathcal{O}(1), T_{\mathbb{P}^3}, \mathcal{O}(2), \mathcal{O}(3))$. We sketch some $\gamma$-commuting and $\gamma$-admissible mutations of $\sigma$, demonstrating some of the diagram in Figure \ref{fig:some-adm-muts-2} in this case.
	
	First, the right mutation of the pair $(\mathcal{O}_{\mathbb{P}^3}(1), T_{\mathbb{P}^3})$ is $\gamma$-admissible, and is defined by a short exact sequence
		\[0 \rightarrow \mathcal{O}_{\mathbb{P}^3}(1) \rightarrow T_{\mathbb{P}^3}^4 \rightarrow E_{15/11} \rightarrow 0,\]
	where the exceptional bundle $E_{\alpha}$ has slope $\alpha$. In the notation of Definition \ref{admissible}, this is $R(F, R_FE)$. In the notation of Figure \ref{fig:some-adm-muts}, this mutation forms the helix $\sigma_2$.
	
	The left mutation of the pair $(T_{\mathbb{P}^3}, \mathcal{O}_{\mathbb{P}^3}(2))$ is also $\gamma$-admissible, and is defined by a short exact sequence
		\[0 \rightarrow E_{22/17} \rightarrow T_{\mathbb{P}^3}^6 \rightarrow \mathcal{O}_{\mathbb{P}^3}(2) \rightarrow 0.\]
	In the notation of Definition \ref{admissible}, this is $L(R_FE, G)$. We will fit this in to Figure \ref{fig:some-adm-muts-2} as $\epsilon_1$ (note that there is some choice in drawing the diagram, so we could have also declared this helix to be $\epsilon_2$). The helix $\epsilon_1$ has foundation $(\mathcal{O}_{\mathbb{P}^3}(1), E_{22/17}, T_{\mathbb{P}^3}, \mathcal{O}_{\mathbb{P}^3}(3))$. 
	
	Note that the mutation $\epsilon_1$ fits into a helix triangle with a $\gamma$-commuting mutation $\gamma'$, called $\tau_1$ in Figure \ref{fig:some-adm-muts-2}. The mutation $\epsilon_1 \mapsto \tau_1$ is given by 
		\[(\mathcal{O}_{\mathbb{P}^3}(1), E_{22/17}, T_{\mathbb{P}^3}, \mathcal{O}_{\mathbb{P}^3}(3)) \mapsto (L_{\mathcal{O}_{\mathbb{P}^3}(1)}E_{22/17}, \mathcal{O}_{\mathbb{P}^3}(1), T_{\mathbb{P}^3}, \mathcal{O}_{\mathbb{P}^3}(3)).\]
	The helix $\tau_1$ obtained this way also has foundation $(\mathcal{O}_{\mathbb{P}^3}(1), T_{\mathbb{P}^3}, \mathcal{O}_{\mathbb{P}^3}(3), (L_{\mathcal{O}_{\mathbb{P}^3}(1)}E_{22/17})(4))$, and the helix condition implies that
		\[(L_{\mathcal{O}_{\mathbb{P}^3}(1)}E_{22/17})(4) \simeq R_{\mathcal{O}_{\mathbb{P}^3}(3)}\mathcal{O}_{\mathbb{P}^3}(2) = T_{\mathbb{P}^3}(2),\]
	and that the mutation $\gamma'$ (in the notation of Figure \ref{fig:some-adm-muts-2}) is the $\gamma$-commuting mutation $R(\mathcal{O}_{\mathbb{P}^3}(2), \mathcal{O}_{\mathbb{P}^3}(3))$.
\end{example}

We now define the subgraphs of the graph of constructive helices $\text{H}_{\PP^3}$. Denote the 8 mutations of the standard helix $\sigma_0$ by $\gamma_i: \sigma_0 \mapsto \sigma_i$, $1 \leq i \leq 8$ (the ordering will not be important for us). These mutations introduce the tangent bundle $T_{\PP^3}$ and cotangent bundle $T_{\PP^3}^{\vee}$ and their twists; see Example \ref{euler}.

\begin{definition} \label{graphs-def}
	We define $\Gamma_{\PP^3}$ inductively by first including the vertex corresponding to $\sigma_0$ and each $\sigma_i$ and an edge corresponding to these mutations. Then for any edge in $\Gamma_{\PP^3}$ corresponding to a mutation $\gamma: \sigma \mapsto \tau$, we add a vertex for each $\gamma$-admissible mutation of $\tau$, each of which is attached to $\tau$ by an edge. 
	
	To define $\Gamma_{\PP^3}'$, we work on explicit foundations for all helices involved. We give $\sigma_0$ the preferred foundation $(\OO(-1), \OO, \OO(1), \OO(2))$, and define all mutations in terms of this foundation. We set $\Gamma_{\PP^3}'$ to be the subgraph of $\Gamma_{\PP^3}$ consisting of helices obtained from $\sigma_0$ via admissible mutations that produce bundles with slope $0 < \mu < 1$ from the preferred foundation of $\sigma_0$.
	
	Both $\Gamma_{\PP^3}$ and its subgraph $\Gamma_{\PP^3}'$ were defined by iterating the process of mutation. For each $k \geq 0$, we set $(\Gamma_{\PP^3})_k$ to be the subgraph of $\Gamma_{\PP^3}$ obtained by the first $k$ mutation steps. For example. $(\Gamma_{\PP^3})_0$ is the singleton vertex corresponding to $\sigma_0$, and $(\Gamma_{\PP^3})_1$ is the graph with nine vertices, one corresponding to $\sigma_0$ which is connected to each of the other eight vertices by a single edge, and no other edges. We set $(\Gamma_{\PP^3}')_k$ to be the intersection $\Gamma_{\PP^3}' \cap (\Gamma_{\PP^3})_k$.
\end{definition}

\begin{remark}
	The graph of all helices and all mutations on $\mathbb{P}^2$ is a tree of degree 3, called $\Gamma_{\PP^2}$ in \cite{BG}.
\end{remark}

\subsection{Admissible replacements} In this subsection we show how to replace arbitrary helices by admissible ones; the main result is Proposition \ref{adm-replacement}. The following construction will be used in the proof of Proposition \ref{adm-replacement}.

\begin{lemma} \label{admissible-lemma}
	Let $\gamma_1: \sigma_1 \mapsto \sigma_2$ be a mutation, and $\gamma_2: \sigma_2 \mapsto \sigma_3$ be the reapplication of $\gamma_1$ to $\sigma_2$. As in Figure \ref{fig:lemma-mutations}, let the helix $\tau_1$ fit into a helix triangle with $\gamma_1$, and $\tau_2$ fit into a helix triangle with $\gamma_2$. 
	
	Then the mutations $\sigma_2 \mapsto \tau_1, \tau_2$ fit into a commuting square with a third helix which we call $\tau_3$, as in Figure \ref{fig:lemma-mutations}. Moreover, the mutation $\tau_1 \mapsto \tau_3$ is admissible with respect to the mutation $\sigma_1 \mapsto \tau_1$. 
\end{lemma}

\begin{figure}[!htb]
	\begin{tikzcd}
		&                                      & \tau_3                                     &                                      &          \\
		& \tau_1 \arrow[rd] \arrow[ru, dotted] &                                            & \tau_2 \arrow[rd] \arrow[lu, dotted] &          \\
		\sigma_1 \arrow[rr, "\gamma_1"] \arrow[ru] &                                      & \sigma_2 \arrow[rr, "\gamma_2"] \arrow[ru] &                                      & \sigma_3
	\end{tikzcd}
	\caption{Mutations for Lemma \ref{admissible-lemma}.}
	\label{fig:lemma-mutations}
\end{figure}
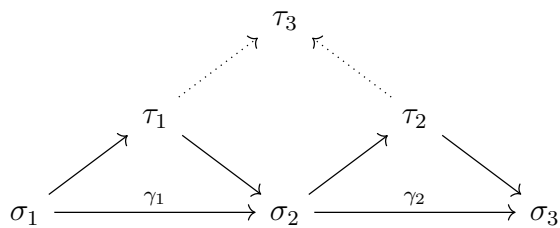

\begin{proof}
	To check the lemma we choose foundations for the helix $\sigma_1$, and make a choice of whether the mutation $\gamma_1$ is a right or left mutation; it will be clear that the proof does not depend on these choices, in the sense that identical arguments can be used when $\gamma_1$ is a left mutation. 
	
	Let $(A,B,C,D)$ be a foundation for $\sigma_1$, so that $\gamma_1: \sigma_1 \mapsto \sigma_2$ is the right mutation $(A,B) \mapsto (B, R_BA)$. Then a foundation for $\sigma_2$ is $(B, R_BA, C, D)$, and a foundation for $\sigma_3$ is $(R_BA, R_{R_BA}B, C, D)$. 
	
	By construction, we have:
	
	\[\text{the helix } \tau_1 \text{ has foundation } (B, C, L_D(A(4)), D)\] 
	
	and 
	
	\[\text{the helix } \tau_2 \text{ has foundation } (R_BA, C, L_D(B(4)), D).\]
	
	 The mutation $\sigma_2 \mapsto \tau_1$ is the right mutation of the pair $(R_BA, C)$, and the mutation $\sigma_2 \mapsto \tau_2$ is the left mutation in the pair $(D, B(4))$. These are disjoint pairs in a foundation for $\sigma_2$, so these mutations fit into a commuting square with one another. We call the fourth helix in this commuting square $\tau_3$.
	
	Checking that the mutation $\tau_1 \mapsto \tau_3$ is admissible with respect to the mutation $\sigma_1 \mapsto \tau_1$ can also be done on foundations. It is convenient to use the foundation $(B,C,D,A(4))$ for $\sigma_1$. With respect to this foundation, the mutation $\sigma_1 \mapsto \tau_1$ is the left mutation in $(D, A(4))$, and $\tau_1$ has foundation $(B,C,L_D(A(4)), D)$. The mutation $\tau_1 \mapsto \tau_3$ is in the pair $(D, B(4))$ (because the mutation $\sigma_2 \mapsto \tau_2$ on the other side of the square was in this pair), and this mutation is admissible; see Definition \ref{admissible} (left mutation case, third entry for admissible mutations).
\end{proof}

\begin{definition} \label{par-mut-def}
	Throughout, we fix a mutation $\gamma: \sigma \mapsto \tau$. 
	
	Let $\gamma': \tau \mapsto \epsilon$ be a mutation of the helix $\tau$, and let $\tau \mapsto \tau'$ be a $\gamma$-commuting mutation. The mutations of $\tau$ fit into commutative squares with other mutations, which we call \textit{parallell mutations}.
	
	Below, we define these mutations of $\tau$ on foundations, and plot them along with some other mutations. The diagrams differ by what kind of mutation $\gamma'$ is. In blue, we plot the parallel mutation $\gamma''$ to $\gamma'$, as well as paths of mutations beginning in $\gamma$ and ending in the parallel mutation.
	
	\begin{enumerate}
		\item 
		$\begin{array}{l} \gamma'\tau = \gamma^2\sigma: \\ \gamma': (F, R_FE, G, H) \mapsto (R_FE, R_{R_FE}F, G, H)) \\ \gamma'': (F, R_FE, L_GH, G) \mapsto (R_FE, R_{R_FE}F, L_GH, G)  \end{array}$ \quad 
		\adjustbox{scale=.7}
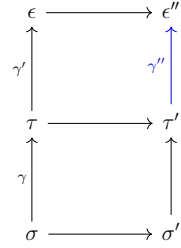
{
			\begin{tikzcd}
				\epsilon \arrow[rr]          &  & \epsilon''         \\
				&  &                    \\
				\tau \arrow[rr] \arrow[uu, "\gamma'"] &  & \tau' \arrow[uu, blue, "\gamma''"] \\
				&  &                    \\
				\sigma \arrow[rr] \arrow[uu, "\gamma"]   &  & \sigma' \arrow[uu]  
		\end{tikzcd}} \\
		
		\item $\begin{array}{l}
			\gamma' \text{ fits in a triangle with } \tau \mapsto \tau': \\
			\gamma': (F, R_FE, G, H) \mapsto (F, R_F(H(-4)), R_FE, G) \\ \gamma'': (F, R_FE, L_GH, G) \mapsto (F, R_F(G(-4)), R_FE, L_GH)
		\end{array}$
		\adjustbox{scale=.7}
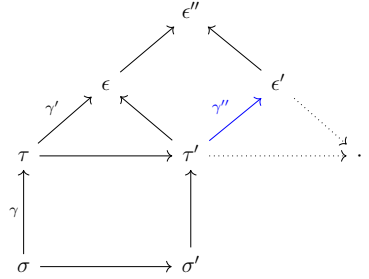
{
			\begin{tikzcd}
				&                     & \epsilon''                                      &                                         &       \\
				& \epsilon \arrow[ru] &                                                  & \epsilon' \arrow[rd, dotted] \arrow[lu] &       \\
				\tau \arrow[rr] \arrow[ru, "\gamma'"] &                     & \tau' \arrow[ru, blue, "\gamma''"] \arrow[rr, dotted] \arrow[lu] &                                         & \cdot \\
				&                     &                                                  &                                         &       \\
				\sigma \arrow[rr] \arrow[uu, "\gamma"]   &                     & \sigma' \arrow[uu]                                 		&                                         &      
		\end{tikzcd}} \\
		
		\item $\begin{array}{l} \gamma' \text{ fits in a triangle with } \gamma \text{ applied to } \tau: \\ \gamma': (F, R_FE, G, H) \mapsto (R_FE, G, L_H(F(4)), H) \\ \gamma'': (E,F,G,H) \mapsto (F, G, L_H(E(4)), H)  \end{array}$ \qquad \quad
		\adjustbox{scale=.7}
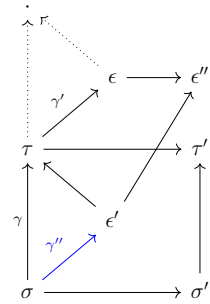
{
			\begin{tikzcd}
				\cdot                                           &                                       &                  \\
				& \epsilon \arrow[lu, dotted] \arrow[r] & \epsilon''       \\
				\tau \arrow[rr] \arrow[uu, dotted] \arrow[ru, "\gamma'"] &                                       & \tau'          \\
				& \epsilon' \arrow[lu] \arrow[ruu]      &                  \\
				\sigma \arrow[rr] \arrow[uu, "\gamma"] \arrow[ru, blue, "\gamma''"]           &                                       & \sigma' \arrow[uu]
		\end{tikzcd}}
	\end{enumerate}
	
	There are three figures. In each, the mutation $\gamma: \sigma \mapsto \tau$ fits into a square with another mutation $\sigma \mapsto \sigma'$. In our descriptions of the parallel mutations on foundations, we need to make a choice of whether $\sigma \mapsto \sigma'$ is a right or left mutation; above we have chosen left. If $\sigma \mapsto \sigma'$ is a right mutation, the diagrams look entirely similar, and the descriptions of the parallel mutations on foundations are analogous to those given above. These describe 6 of the 8 total mutations of $\tau$.
	
	The remaining two mutations are not assigned parallel mutations. These are (i) the inverse mutation $\gamma^{-1}: \tau \mapsto \sigma$ of $\gamma$, and (ii) the mutation of $\tau$ which fits into a helix triangle with $\gamma$ (in (3) above, this is the mutation $\tau \mapsto \epsilon'$).
\end{definition}

In the diagrams, all squares are commutative and all triangles are helix triangles, with dots in place of helices which are unimportant but illustrative, and dotted arrows marking mutations which connect these to the relevant helices. 

\begin{remark} \label{par-mut-rem}
	The parallel mutation $\gamma''$ of a mutation $\gamma'$ produce new target helices $\epsilon''$. In all cases the helix $\epsilon''$ contains the bundle produced by the mutation $\gamma'$. 
	
	The definition is formed so as to be suitable for inductive arguments. In Proposition \ref{adm-replacement}, we will show that sequence of mutations $\sigma_0 \mapsto \sigma_1 \mapsto \cdots \mapsto \sigma$ which produces an exceptional bundle $E$ in the final step, can be replaced by a sequence $\sigma_0 \mapsto \sigma_1' \mapsto \cdots \sigma'$ which also produces $E$, with the property that each mutation $\sigma_{k-1}' \mapsto \sigma_k'$ is admissible with respect to the prior mutation in the new sequence. The parallel mutations to the mutations $\sigma_{k-1} \mapsto \sigma_k$ will be the central tool for forming these replacements.
	
	See Example \ref{adm-rep-ex} for an example of a parallel sequence.
\end{remark}

Because parallel mutations form squares in which the original mutation sits, we can iterate the process of forming parallel mutations. 

\begin{definition} \label{par-seq-def}
	If $\gamma: \sigma \mapsto \tau$ is a mutation, $\gamma': \sigma' \mapsto \tau'$ is a parallel mutation, and $\tau \mapsto \tau_1 \mapsto \cdots \mapsto \tau_k$ is a sequence of mutations, then the sequence $\tau' \mapsto \tau_1' \mapsto \cdots \mapsto \tau'_k$ obtained by taking parallel mutations at each step is called the \textit{parallel sequence} to the given sequence.
\end{definition}

We will use the notion of parallel mutations to replace arbitrary paths on the graph $\text{H}_{\mathbb{P}^3}$ with admissible paths.

\begin{proposition} \label{adm-replacement}
	Let $E$ be a constructive exceptional bundle on $\mathbb{P}^3$, and let $\sigma$ be a constructive helix containing $E$, with a sequence of mutations $\sigma_0 \mapsto \sigma_1 \mapsto \cdots \mapsto \sigma_n = \sigma$ realizing constructivity of $\sigma$ of minimal length among helices containing $E$. 
	
	Then there is another constructive helix $\sigma'$ also containing $E$, with a sequence of admissible mutations $\sigma_0 \mapsto \sigma_1' \mapsto \cdots \mapsto \sigma_{n'}' = \sigma'$. 
\end{proposition}

\begin{proof}
	Write $\gamma_i: \sigma_{i-1} \mapsto \sigma_i$ for the mutations in the given sequence, and write $\gamma = \gamma_1 \circ \cdots \circ \gamma_N$ for the total mutation. Let $k = k_{\gamma}$ be the smallest index such that $\gamma_{N-k}$ is not admissible with respect to $\gamma_{N-k-1}$. 
	
	
	Note that $\gamma_{N-k}$ is either $\gamma_{N-k-1}$-commuting or fits into a helix triangle with the reapplication of $\gamma_{N-k-1}$ to $\sigma_{N-k-1}$; this is because the other non-admissible mutations of $\sigma_{N-k-1}$ don't minimize the total length of $\gamma$. 
	
	
	We will replace $\gamma = \gamma_1 \circ \cdots \circ \gamma_N$ by a sequence 
		\[\gamma' = \gamma_1 \circ \cdots \circ \gamma_{N-k-2} \circ \gamma_{N-k-1}' \circ \gamma_{N-k}' \circ \cdots \circ \gamma_N'\]
	of mutations of the same length, such that the last mutation $\gamma_N'$ still produces a helix containing $E$, and we have $k_{\gamma'} > k_{\gamma} = k$. 
	
	To do so, we induct on $k$. The base case differs based on the type of non-admissible mutation which occurs. If the non-admissible mutation $\gamma_{N-k}$ is commuting, then the base case is where $k = 1$ since if $k = 0$, then the final mutation $\gamma_N$ is commuting with respect to $\gamma_{N-1}$, so the mutation along the other side of the square formed by $\gamma_{N-1}$ and $\gamma_N$ also produces $E$, so $\gamma$ is not of minimal length. If $\gamma_{N-k}$ is not commuting, then the base case is $k = 0$.
	
	Suppose that $\gamma_{N-k}$ is commuting, so that $k = 1$ and $\gamma_{N-1}$ is commuting with respect to $\gamma_{N-2}$, and $\gamma_N$ is admissible with respect to $\gamma_{N-1}$. There are three $\gamma_{N-1}$-admissible mutations, but we will check that only one of them can occur as $\gamma_N$. See Figure \ref{fig:B}. 
	
	\begin{figure}[!htb]
		\begin{tikzcd}
			&  & \sigma'''                                                       &                           \\
			&  &                                                                 & \tau_1 \arrow[lu, dotted] \\
			\sigma_{N-2} \arrow[rr, "\gamma_{N-1}"]     &  & \sigma_{N-1} \arrow[rd] \arrow[r] \arrow[ru] \arrow[uu, dotted] & \tau_2                    \\
			&  &                                                                 & \tau_3                    \\
			\sigma_{N-3} \arrow[uu, "\gamma_{N-2}"] \arrow[rr, "\gamma_{N-2}'"] &  & \sigma' \arrow[ru, "\gamma_{N-1}'", swap] \arrow[r] \arrow[uu]                         & \sigma''                 
		\end{tikzcd}
		\caption{The commuting base case for Proposition \ref{adm-replacement}.}
		\label{fig:B}
	\end{figure}
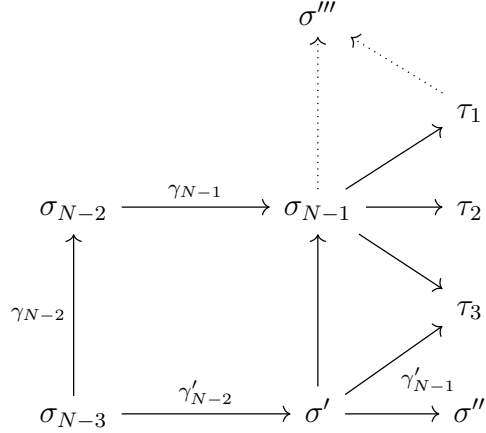

	In Figure \ref{fig:B}, the helices $\tau_1, \tau_2$, and $\tau_3$ are $\gamma_{N-1}$-admissible. The helix $\sigma'$ is the unique helix sitting in a commuting square with $\sigma_{N-3}$, $\sigma_{N-2}$, and $\sigma_{N-1}$. The helix $\sigma''$ is parallel to $\tau_2$, which is formed from $\sigma_{N-1}$ by reapplying $\gamma_{N-1}$. The helix $\tau_3$ fits into a helix triangle with $\sigma'$, and the helix $\tau_1$ fits into a helix triangle with the helix $\sigma'''$. 
	
	The helix $\sigma''$ contains the bundle produced by the mutation $\sigma_{N-1} \mapsto \tau_2$, so if the mutation $\gamma_N$ is the mutation $\sigma_{N-1} \mapsto \tau_2$, then $\gamma$ can be replaced by a shorter sequence of mutations ending with the mutations $\sigma_{N-3} \mapsto \sigma' \mapsto \sigma''$. Likewise, the mutation $\sigma_{N-1} \mapsto \tau_3$ can be replaced by the mutations $\sigma_{N-3} \mapsto \sigma' \mapsto \tau_3$, so in this case as well, $\gamma$ would not be of minimal length. We conclude that $\gamma_N$ is the mutation $\sigma_{N-1} \mapsto \tau_1$. It remains to replace $\gamma_{N-1} \circ \gamma_N$ by a sequence of admissible mutations. To do so, we extend Figure \ref{fig:B} in Figure \ref{fig:C}.
	
	
	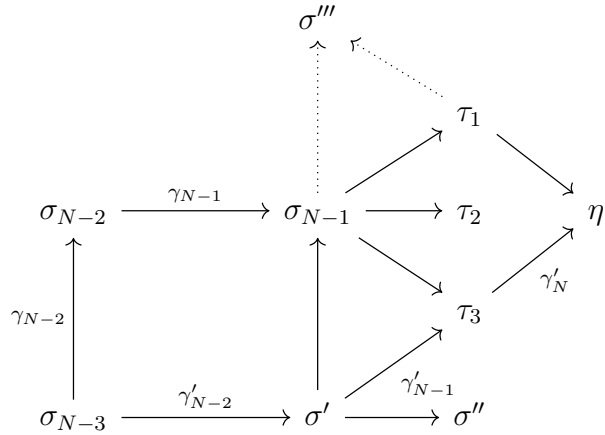
\begin{figure}[!htb]
		\begin{tikzcd}
			&  & \sigma'''                                                       &                                      &      \\
			&  &                                                                 & \tau_1 \arrow[lu, dotted] \arrow[rd] &      \\
			\sigma_{N-2} \arrow[rr, "\gamma_{N-1}"]     &  & \sigma_{N-1} \arrow[rd] \arrow[r] \arrow[ru] \arrow[uu, dotted] & \tau_2                               & \eta \\
			&  &                                                                 & \tau_3 \arrow[ru, "\gamma_N'", swap]                    &      \\
			\sigma_{N-3} \arrow[uu, "\gamma_{N-2}"] \arrow[rr, "\gamma_{N-2}'"] &  & \sigma' \arrow[ru, "\gamma_{N-1}'", swap] \arrow[r] \arrow[uu]                         & \sigma''                             &     
	\end{tikzcd}
		\caption{Extension of Figure \ref{fig:B} for Proposition \ref{adm-replacement}.}
		\label{fig:C}
	\end{figure}
	
	To define the helix $\eta$, we appeal to Lemma \ref{admissible-lemma}, which implies that the mutations $\sigma_{N-1} \mapsto \tau_1, \tau_3$ sit in a commuting square, and $\eta$ is defined to be the helix filling in the square. Because this square is commuting, the bundle produced by $\gamma_N$ (which is $E$) also appears in $\eta$. In Lemma \ref{admissible-lemma} we also checked that the mutation $\tau_3 \mapsto \eta$ is admissible with respect to $\gamma_{N-1}': \sigma' \mapsto \tau_3$. The sequence of mutations $\gamma_{N-2}' \circ \gamma_{N-1}' \circ \gamma_N'$ is the desired replacement of the tail of $\gamma$ in this case.
	
	Now suppose that $\gamma_{N-k}$ is non-commuting. Then the base case has $k = 0$, and $\gamma_N$ fits into a helix triangle with the reapplication of $\gamma_{N-1}$ to $\sigma_{N-1}$. See Figure \ref{fig:D}. 
	
	\begin{figure}[!htb]
		\begin{tikzcd}
			&  &                                    & \sigma_N \arrow[rd] &       \\
			\sigma_{N-2} \arrow[rr, "\gamma_{N-1}"] &  & \sigma_{N-1} \arrow[rr] \arrow[ru, "\gamma_N"] &                     & \cdot
		\end{tikzcd}
		\caption{The non-commuting base case for Proposition \ref{adm-replacement}.}
		\label{fig:D}
	\end{figure}
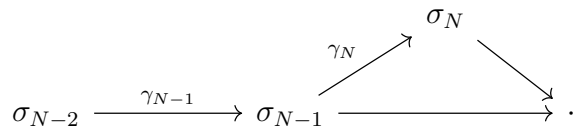

	As in Lemma \ref{admissible-lemma}, the mutation $\gamma_N$ fits into a commuting square with the mutation $\gamma_{N-1}'$, in the notation of Figure \ref{fig:D}. The new helix $\sigma_N'$ also contains $E$ as in the previous case. Moreover, Lemma \ref{admissible-lemma} also implies that $\gamma_N'$ is $\gamma_{N-1}'$-admissible. We extend Figure \ref{fig:D} with this additional set of mutations in Figure \ref{fig:E}. The pair of mutations $(\gamma_{N-1}', \gamma_N')$ replaces $(\gamma_{N-1}, \gamma_N)$ to form $\gamma'$. This completes the base case.
	
	\begin{figure}[!htb]
		\begin{tikzcd}
			&                                     & \sigma_N' \arrow[rd]               &                     &       \\
			& \sigma_{N-1}' \arrow[rd] \arrow[ru, "\gamma_N'"] &                                    & \sigma_N \arrow[rd] &       \\
			\sigma_{N-2} \arrow[rr, "\gamma_{N-1}"] \arrow[ru, "\gamma_{N-1}'"] &                                     & \sigma_{N-1} \arrow[rr] \arrow[ru, "\gamma_N"] &                     & \cdot
		\end{tikzcd}
		\caption{Extension of Figure \ref{fig:D} for Proposition \ref{adm-replacement}.}
		\label{fig:E}
	\end{figure}
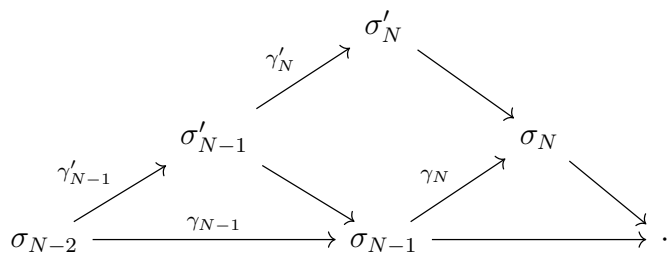
	
	We now carry out the inductive step, to show how a sequence $\gamma$ can be replaced by a new sequence $\gamma'$ with $k_{\gamma'} > k_{\gamma} = k$. It is entirely similar to the previous replacements. Then in at most $N$ steps, we arrive at a replacement for $\gamma$ in which each mutation is admissible. Again, the first non-admissible mutation $\gamma_{N-k}$ in the sequence is either $\gamma_{N-k-1}$-commuting, or fits in a helix triangle with the reapplication of $\gamma_{N-k-1}$. 
	
	First, suppose that $\gamma_{N-k}$ is $\gamma_{N-k-1}$-commuting. As in Figure \ref{fig:C}, the mutation $\gamma_{N-k+1}$ can be one of three possible admissible mutations, but only one of these is possible since $\gamma$ is assumed to have minimal length. As in the base case, we can form a square containing a new helix, such that the tail of the replacement sequence has exactly one more admissible mutation than does $\gamma$. To form a sequence of mutations that forms $E$, we extend this new subsequence with a sequence parallel to $\gamma$, as in Definition \ref{par-seq-def}. Each step in the parallel sequence differs from a $\sigma_j$ by a commuting square, so $E$ still appears in the final mutation of the parallel sequence.
	
	When $\gamma_{N-k}$ is non-$\gamma_{N-k-1}$-commuting, it fits into a helix triangle with the reapplication of $\gamma_{N-k-1}$. As in Figure \ref{fig:E}, we can again form a commuting square to replace $\gamma_{N-k}$ by a subsequence of mutations with exactly one more admissible mutation than $\gamma$. We can then extend this subsequence along a sequence parallel to $\gamma$ to finish the induction.
\end{proof}

It follows from Proposition \ref{adm-replacement} that every constructive exceptional bundle $E$ appears in an admissible helix $\sigma$ appearing in the subgraph $\Gamma_{\mathbb{P}^3}$ of $\text{H}_{\mathbb{P}^3}$. 

\begin{example} \label{adm-rep-ex}
We give an explicit example of the replacement of a sequence of non-admissible mutations by admissible ones to illustrate Proposition \ref{adm-replacement}. The bundle we will produce has slope $195/59$. (We will later recognize it as a twist of the bundle $E_{18/59}$.) In Figure \ref{fig:ex-fig-a} we give a diagram of mutations which produce this bundle.

\begin{figure}[!htb]
	\begin{tikzcd}
		&  & {(T, E_{11/15}, \mathcal{O}(3), T(2))}                                    &                                                           \\
		&  &                                                                           & {(T, \mathcal{O}(3), E_{195/59},T(2))} \arrow[lu, dotted] \\
		{(\mathcal{O}(1), T, \mathcal{O}(2), \mathcal{O}(3))} \arrow[rr] &  & {(\mathcal{O}(1), T, \mathcal{O}(3), T(2))} \arrow[uu, dotted] \arrow[ru] &                                                           \\
		&  &                                                                           &                                                           \\
		\sigma_0 \arrow[uu] \arrow[rr, dotted]                           &  & {(\mathcal{O}, \mathcal{O}(1), \mathcal{O}(3), T(2))} \arrow[uu, dotted]  &                                                          
	\end{tikzcd}
	\caption{Mutations for Example \ref{adm-rep-ex}, producing an exceptional bundle of slope $195/59$.}
	\label{fig:ex-fig-a}
\end{figure}
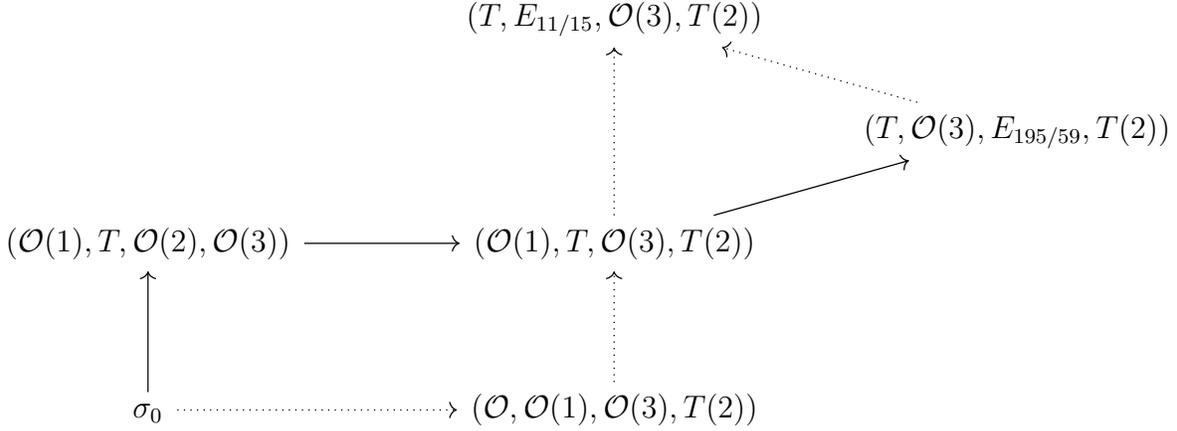

The dashed mutations are included for clarity; the dark mutations are the ones we consider. Specifically, the first mutation is the formation of the tangent bundle (see Example \ref{euler}). The second is the formation of the twist $T(2)$ of the tangent bundle, and the third is the left mutation of the pair $(T(2), \mathcal{O}(5))$ defined by the short exact sequence
	\[0 \rightarrow E_{195/59} \rightarrow T(2)^{20} \rightarrow \mathcal{O}(5) \rightarrow 0.\]

The second mutation is not admissible with respect to the first; instead, it is commuting. As in the proof of Proposition \ref{adm-replacement}, we can extend Figure \ref{fig:ex-fig-a} using helix triangles and commuting squares in Figure \ref{fig:ex-fig-b} below.

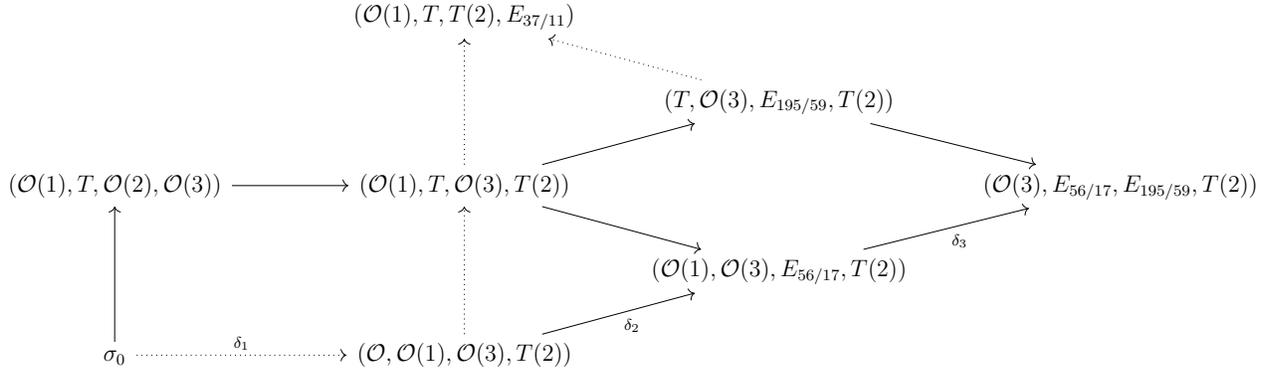
\begin{figure}[!htb]
	\adjustbox{scale = .75,center}{
	\begin{tikzcd}
		&  & {(\mathcal{O}(1), T, T(2), E_{37/11})}                                                           &                                                                             &                                               \\
		&  &                                                                                                  & {(T, \mathcal{O}(3), E_{195/59},T(2))} \arrow[lu, dotted] \arrow[rd]        &                                               \\
		{(\mathcal{O}(1), T, \mathcal{O}(2), \mathcal{O}(3))} \arrow[rr] &  & {(\mathcal{O}(1), T, \mathcal{O}(3), T(2))} \arrow[uu, dotted] \arrow[ru] \arrow[rd]             &                                                                             & {(\mathcal{O}(3), E_{56/17},E_{195/59},T(2))} \\
		&  &                                                                                                  & {(\mathcal{O}(1), \mathcal{O}(3), E_{56/17}, T(2))} \arrow[ru, "\delta_3"'] &                                               \\
		\sigma_0 \arrow[uu] \arrow[rr, "\delta_1", dotted]               &  & {(\mathcal{O}, \mathcal{O}(1), \mathcal{O}(3), T(2))} \arrow[uu, dotted] \arrow[ru, "\delta_2"'] &                                                                             &                                              
	\end{tikzcd}}
	\caption{Admissible replacement for Example \ref{adm-rep-ex}.}
	\label{fig:ex-fig-b}
\end{figure}

As in the proof of Proposition \ref{adm-replacement}, the sequence of mutations $\delta_3 \circ \delta_2 \circ \delta_1$ is admissible. The mutation $\delta_1$ is the standard formation of the twist $T(2)$ of the tangent bundle. The mutation $\delta_2$ is the left mutation of the pair $(T(2), \mathcal{O}(4))$. It is defined by a short exact sequence

\[0 \rightarrow E_{56/17} \rightarrow T(2)^6 \rightarrow \mathcal{O}(4) \rightarrow 0.\]

The mutation $\delta_3$ is the left mutation of the pair $(T(2), \mathcal{O}(5))$. It is defined by a short exact sequence

\[0 \rightarrow E_{195/59} \rightarrow T(2)^{20} \rightarrow \mathcal{O}(5) \rightarrow 0.\]

This is the same exact sequence that formed $E_{195/59}$ in the first sequence of mutations; this is because the pair $(T(2), \mathcal{O}(5))$ appears in both helices. This in turn is because these helices sit in a commuting square with one another.

In the terminology of Definition \ref{par-mut-def}, the mutations $\delta_3 \circ \delta_2 \circ \delta_1$ is parallel to the original sequence of mutations
\end{example}

\subsection{The graph $\Gamma_{\mathbb{P}^3}'$} Given a mutation $\gamma: \sigma \mapsto \tau$, there are exactly three $\gamma$-admissible mutations of $\tau$, which depend on whether $\gamma$ was a left or right mutation. In Proposition \ref{gamma-tree}, we will describe the helices obtained from admissible mutations, and it will be convenient to establish some extra notation for admissible mutations.

Choose a foundation $(E,F,G,H)$ for $\tau$. If $\gamma$ is a right mutation, we let $F$ denote the bundle produced by $\gamma$, and if $\gamma$ is a left mutation, we let $G$ denote the bundle produced by $\gamma$. E.g., we will write the helix obtained from $\sigma_0$ which contains the tangent bundle as $(\mathcal{O}(1), T_{\mathbb{P}^3}, \mathcal{O}(2), \mathcal{O}(3))$ because $T_{\mathbb{P}^3}$ is a right mutation, and we will write the helix obtained from $\sigma_0$ which contains $T^{\vee}_{\mathbb{P}^3}$ as $(\mathcal{O}(-3), \mathcal{O}(-2), T^{\vee}_{\mathbb{P}^3}, \mathcal{O}(-1))$ because $T^{\vee}_{\mathbb{P}^3}$ is a left mutation.

We isolate the three $\gamma$-admissible mutations of $\tau$ in terms of this foundation. For $\gamma$ a right mutation, we define the mutations $R0$, $L0$, and $R1$ of $\tau$, and for $\gamma$ a left mutation, we define the mutations $L1$, $R2$, and $L2$ of $\tau$, as follows.

\begin{align*}
	R0: (E,F,G,H) &\mapsto (E, R_E (H(-4)), F, G) \quad &&L0: (E,F,G,H) \mapsto (H(-4), E, L_FG, F) \\
	R1: (E,F,G,H) &\mapsto (F, R_FE, G, H) \quad &&L1: (E,F,G,H) \mapsto (E,F, L_GH, G) \\
	R2: (E,F,G,H) &\mapsto (G, R_GF, H, E(4)) \quad &&L2: (E,F,G,H) \mapsto (F, G, L_H(E(4)),H)
\end{align*}

Note that in each mutation, the target mutation is again expressed with the new bundle on the left for right mutations, and on the right for left mutations. See Figure \ref{fig:adm-graph} for a drawing, where the vertices are ordered left-to-right by the slope of the bundle introduced by the mutation. After fixing these slopes, the right mutation helix (the left dot in the middle row between 0 and 1) contains $T_{\mathbb{P}^3}(-1)$, and the left mutation helix (the right dot in the middle row between 0 and 1) contains $T^{\vee}_{\mathbb{P}^3}(2)$.

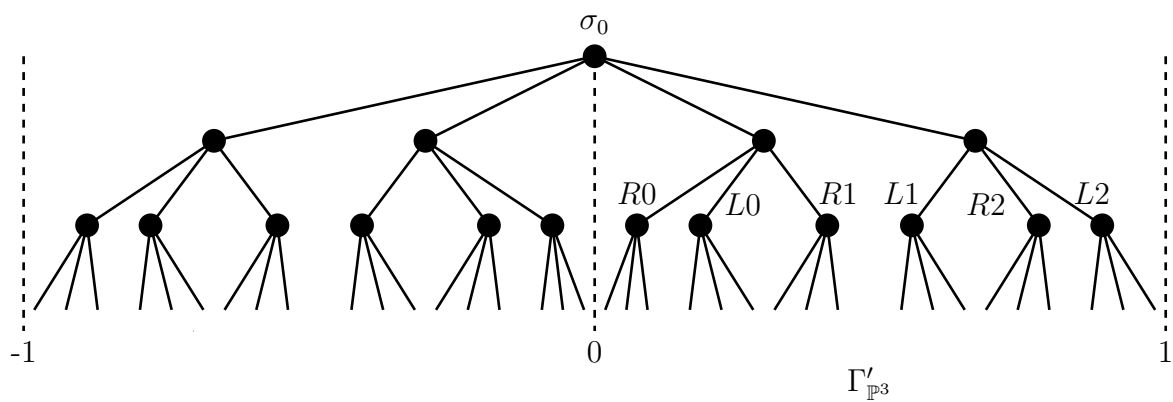
\begin{figure}[!htb]
	\centering
	\begin{tikzpicture}[x=0.4pt, y=0.4pt]
		\draw[solid, draw={rgb,255:red,0;green,0;blue,0}, draw opacity=1, line width=1, fill={rgb,255:red,0;green,0;blue,0}, fill opacity=1] (690,670) ellipse (10 and 10);
		\draw[solid, draw={rgb,255:red,0;green,0;blue,0}, draw opacity=1, line width=1, fill={rgb,255:red,0;green,0;blue,0}, fill opacity=1] (530,590) ellipse (0 and 0);
		\draw[solid, draw={rgb,255:red,0;green,0;blue,0}, draw opacity=1, line width=1, fill={rgb,255:red,0;green,0;blue,0}, fill opacity=1] (530,590) ellipse (10 and 10);
		\draw[solid, draw={rgb,255:red,0;green,0;blue,0}, draw opacity=1, line width=1, fill={rgb,255:red,0;green,0;blue,0}, fill opacity=1] (850,590) ellipse (10 and 10);
		\draw[solid, draw={rgb,255:red,0;green,0;blue,0}, draw opacity=1, line width=1, ] (530,590) -- (690,670);
		\draw[solid, draw={rgb,255:red,0;green,0;blue,0}, draw opacity=1, line width=1, ] (690,670) -- (850,590);
		\draw[solid, draw={rgb,255:red,0;green,0;blue,0}, draw opacity=1, line width=1, fill={rgb,255:red,0;green,0;blue,0}, fill opacity=1] (590,510) ellipse (10 and 10);
		\draw[solid, draw={rgb,255:red,0;green,0;blue,0}, draw opacity=1, line width=1, fill={rgb,255:red,0;green,0;blue,0}, fill opacity=1] (650,510) ellipse (10 and 10);
		\draw[solid, draw={rgb,255:red,0;green,0;blue,0}, draw opacity=1, line width=1, fill={rgb,255:red,0;green,0;blue,0}, fill opacity=1] (730,510) ellipse (0 and 0);
		\draw[solid, draw={rgb,255:red,0;green,0;blue,0}, draw opacity=1, line width=1, fill={rgb,255:red,0;green,0;blue,0}, fill opacity=1] (730,510) ellipse (10 and 10);
		\draw[solid, draw={rgb,255:red,0;green,0;blue,0}, draw opacity=1, line width=1, fill={rgb,255:red,0;green,0;blue,0}, fill opacity=1] (790,510) ellipse (10 and 10);
		\draw[solid, draw={rgb,255:red,0;green,0;blue,0}, draw opacity=1, line width=1, fill={rgb,255:red,0;green,0;blue,0}, fill opacity=1] (470,510) ellipse (10 and 10);
		\draw[solid, draw={rgb,255:red,0;green,0;blue,0}, draw opacity=1, line width=1, fill={rgb,255:red,0;green,0;blue,0}, fill opacity=1] (910,510) ellipse (10 and 10);
		\draw[solid, draw={rgb,255:red,0;green,0;blue,0}, draw opacity=1, line width=1, ] (530,590) -- (470,510);
		\draw[solid, draw={rgb,255:red,0;green,0;blue,0}, draw opacity=1, line width=1, ] (530,590) -- (590,510);
		\draw[solid, draw={rgb,255:red,0;green,0;blue,0}, draw opacity=1, line width=1, ] (530,590) -- (650,510);
		\draw[solid, draw={rgb,255:red,0;green,0;blue,0}, draw opacity=1, line width=1, ] (850,590) -- (730,510);
		\draw[solid, draw={rgb,255:red,0;green,0;blue,0}, draw opacity=1, line width=1, ] (850,590) -- (790,510);
		\draw[solid, draw={rgb,255:red,0;green,0;blue,0}, draw opacity=1, line width=1, ] (850,590) -- (910,510);
		\draw[solid, draw={rgb,255:red,0;green,0;blue,0}, draw opacity=1, line width=1, fill={rgb,255:red,0;green,0;blue,0}, fill opacity=1] (390,510) ellipse (10 and 10);
		\draw[solid, draw={rgb,255:red,0;green,0;blue,0}, draw opacity=1, line width=1, fill={rgb,255:red,0;green,0;blue,0}, fill opacity=1] (270,510) ellipse (10 and 10);
		\draw[solid, draw={rgb,255:red,0;green,0;blue,0}, draw opacity=1, line width=1, fill={rgb,255:red,0;green,0;blue,0}, fill opacity=1] (210,510) ellipse (10 and 10);
		\draw[solid, draw={rgb,255:red,0;green,0;blue,0}, draw opacity=1, line width=1, fill={rgb,255:red,0;green,0;blue,0}, fill opacity=1] (330,590) ellipse (10 and 10);
		\draw[solid, draw={rgb,255:red,0;green,0;blue,0}, draw opacity=1, line width=1, ] (330,590) -- (390,510);
		\draw[solid, draw={rgb,255:red,0;green,0;blue,0}, draw opacity=1, line width=1, ] (330,590) -- (270,510);
		\draw[solid, draw={rgb,255:red,0;green,0;blue,0}, draw opacity=1, line width=1, ] (330,590) -- (210,510);
		\draw[solid, draw={rgb,255:red,0;green,0;blue,0}, draw opacity=1, line width=1, ] (330,590) -- (690,670);
		\draw[solid, draw={rgb,255:red,0;green,0;blue,0}, draw opacity=1, line width=1, fill={rgb,255:red,0;green,0;blue,0}, fill opacity=1] (990,510) ellipse (10 and 10);
		\draw[solid, draw={rgb,255:red,0;green,0;blue,0}, draw opacity=1, line width=1, fill={rgb,255:red,0;green,0;blue,0}, fill opacity=1] (1110,510) ellipse (10 and 10);
		\draw[solid, draw={rgb,255:red,0;green,0;blue,0}, draw opacity=1, line width=1, fill={rgb,255:red,0;green,0;blue,0}, fill opacity=1] (1170,510) ellipse (10 and 10);
		\draw[solid, draw={rgb,255:red,0;green,0;blue,0}, draw opacity=1, line width=1, fill={rgb,255:red,0;green,0;blue,0}, fill opacity=1] (1050,590) ellipse (10 and 10);
		\draw[solid, draw={rgb,255:red,0;green,0;blue,0}, draw opacity=1, line width=1, ] (990,510) -- (1050,590);
		\draw[solid, draw={rgb,255:red,0;green,0;blue,0}, draw opacity=1, line width=1, ] (1050,590) -- (1110,510);
		\draw[solid, draw={rgb,255:red,0;green,0;blue,0}, draw opacity=1, line width=1, ] (1050,590) -- (1170,510);
		\draw[solid, draw={rgb,255:red,0;green,0;blue,0}, draw opacity=1, line width=1, ] (1050,590) -- (690,670);
		\draw[solid, draw={rgb,255:red,0;green,0;blue,0}, draw opacity=1, line width=1, ] (160,430) -- (210,510);
		\draw[solid, draw={rgb,255:red,0;green,0;blue,0}, draw opacity=1, line width=1, ] (210,510) -- (190,430);
		\draw[solid, draw={rgb,255:red,0;green,0;blue,0}, draw opacity=1, line width=1, ] (210,510) -- (220,430);
		\draw[solid, draw={rgb,255:red,0;green,0;blue,0}, draw opacity=1, line width=1, ] (270,510) -- (260,430);
		\draw[solid, draw={rgb,255:red,0;green,0;blue,0}, draw opacity=1, line width=1, ] (270,510) -- (290,430);
		\draw[solid, draw={rgb,255:red,0;green,0;blue,0}, draw opacity=1, line width=1, ] (270,510) -- (320,430);
		\draw[solid, draw={rgb,255:red,0;green,0;blue,0}, draw opacity=1, line width=1, ] (340,430) -- (390,510);
		\draw[solid, draw={rgb,255:red,0;green,0;blue,0}, draw opacity=1, line width=1, ] (390,510) -- (370,430);
		\draw[solid, draw={rgb,255:red,0;green,0;blue,0}, draw opacity=1, line width=1, ] (400,430) -- (390,510);
		\draw[solid, draw={rgb,255:red,0;green,0;blue,0}, draw opacity=1, line width=1, fill={rgb,255:red,0;green,0;blue,0}, fill opacity=1] (310,410) ellipse (0 and 0);
		\draw[solid, draw={rgb,255:red,0;green,0;blue,0}, draw opacity=1, line width=1, ] (470,510) -- (460,430);
		\draw[solid, draw={rgb,255:red,0;green,0;blue,0}, draw opacity=1, line width=1, ] (470,510) -- (490,430);
		\draw[solid, draw={rgb,255:red,0;green,0;blue,0}, draw opacity=1, line width=1, ] (470,510) -- (520,430);
		\draw[solid, draw={rgb,255:red,0;green,0;blue,0}, draw opacity=1, line width=1, ] (590,510) -- (600,430);
		\draw[solid, draw={rgb,255:red,0;green,0;blue,0}, draw opacity=1, line width=1, ] (590,510) -- (570,430);
		\draw[solid, draw={rgb,255:red,0;green,0;blue,0}, draw opacity=1, line width=1, ] (590,510) -- (540,430);
		\draw[solid, draw={rgb,255:red,0;green,0;blue,0}, draw opacity=1, line width=1, ] (650,510) -- (640,430);
		\draw[solid, draw={rgb,255:red,0;green,0;blue,0}, draw opacity=1, line width=1, ] (650,510) -- (660,430);
		\draw[solid, draw={rgb,255:red,0;green,0;blue,0}, draw opacity=1, line width=1, ] (650,510) -- (680,430);
		\draw[solid, draw={rgb,255:red,0;green,0;blue,0}, draw opacity=1, line width=1, ] (700,430) -- (730,510);
		\draw[solid, draw={rgb,255:red,0;green,0;blue,0}, draw opacity=1, line width=1, ] (730,510) -- (720,430);
		\draw[solid, draw={rgb,255:red,0;green,0;blue,0}, draw opacity=1, line width=1, ] (730,510) -- (740,430);
		\draw[solid, draw={rgb,255:red,0;green,0;blue,0}, draw opacity=1, line width=1, ] (790,510) -- (780,430);
		\draw[solid, draw={rgb,255:red,0;green,0;blue,0}, draw opacity=1, line width=1, ] (790,510) -- (810,430);
		\draw[solid, draw={rgb,255:red,0;green,0;blue,0}, draw opacity=1, line width=1, ] (790,510) -- (840,430);
		\draw[solid, draw={rgb,255:red,0;green,0;blue,0}, draw opacity=1, line width=1, ] (910,510) -- (920,430);
		\draw[solid, draw={rgb,255:red,0;green,0;blue,0}, draw opacity=1, line width=1, ] (910,510) -- (890,430);
		\draw[solid, draw={rgb,255:red,0;green,0;blue,0}, draw opacity=1, line width=1, ] (910,510) -- (860,430);
		\draw[solid, draw={rgb,255:red,0;green,0;blue,0}, draw opacity=1, line width=1, ] (990,510) -- (980,430);
		\draw[solid, draw={rgb,255:red,0;green,0;blue,0}, draw opacity=1, line width=1, ] (990,510) -- (1010,430);
		\draw[solid, draw={rgb,255:red,0;green,0;blue,0}, draw opacity=1, line width=1, ] (990,510) -- (1040,430);
		\draw[solid, draw={rgb,255:red,0;green,0;blue,0}, draw opacity=1, line width=1, ] (1110,510) -- (1120,430);
		\draw[solid, draw={rgb,255:red,0;green,0;blue,0}, draw opacity=1, line width=1, ] (1110,510) -- (1090,430);
		\draw[solid, draw={rgb,255:red,0;green,0;blue,0}, draw opacity=1, line width=1, ] (1110,510) -- (1060,430);
		\draw[solid, draw={rgb,255:red,0;green,0;blue,0}, draw opacity=1, line width=1, ] (1170,510) -- (1160,430);
		\draw[solid, draw={rgb,255:red,0;green,0;blue,0}, draw opacity=1, line width=1, ] (1170,510) -- (1190,430);
		\draw[solid, draw={rgb,255:red,0;green,0;blue,0}, draw opacity=1, line width=1, ] (1170,510) -- (1220,430);
		\draw[dashed, draw={rgb,255:red,0;green,0;blue,0}, draw opacity=1, line width=1, ] (690,670) -- (690,410);
		\draw[dashed, draw={rgb,255:red,0;green,0;blue,0}, draw opacity=1, line width=1, ] (1230,670) -- (1230,410);
		\draw[dashed, draw={rgb,255:red,0;green,0;blue,0}, draw opacity=1, line width=1, ] (150,670) -- (150,410);
		\node at (150,390) [opacity=1] {\textcolor[RGB]{0,0,0}{-1}};
		\node at (690,390) [opacity=1] {\textcolor[RGB]{0,0,0}{0}};
		\node at (1230,390) [opacity=1] {\textcolor[RGB]{0,0,0}{1}};
		\node at (950,360) [opacity=1] {\textcolor[RGB]{0,0,0}{$\Gamma_{\mathbb{P}^3}'$}};
		\node at (730,540) [opacity=1] {\textcolor[RGB]{0,0,0}{$R0$}};
		\node at (830,530) [opacity=1] {\textcolor[RGB]{0,0,0}{$L0$}};
		\node at (920,540) [opacity=1] {\textcolor[RGB]{0,0,0}{$R1$}};
		\node at (980,540) [opacity=1] {\textcolor[RGB]{0,0,0}{$L1$}};
		\node at (1060,530) [opacity=1] {\textcolor[RGB]{0,0,0}{$R2$}};
		\node at (1160,540) [opacity=1] {\textcolor[RGB]{0,0,0}{$L2$}};
		\node at (690,700) [opacity=1] {\textcolor[RGB]{0,0,0}{$\sigma_0$}};
	\end{tikzpicture}
	
	\caption{Cartoon of the admissible mutations for left and right mutation helices, ordered by the slope of the new bundle produced. In the figure, the vertices correspond to helices and the edges to mutations. }
	\label{fig:adm-graph}
\end{figure}

Consider a helix $\sigma$ corresponding to a vertex $v$ in $\Gamma_{\PP^3}'$. Let $\sigma_1, \sigma_2$, and $\sigma_3$ be its three admissible mutations, and set $\gamma_i: \sigma \mapsto \sigma_i$ to be the mutation producing $\sigma_i$ from $\sigma$. Each of the $\sigma_j$ also correspond to vertices $v_j$ in $\Gamma_{\PP^3}'$. It follows from Lemma \ref{mutations} that the $\sigma_j$ are pairwise distinct. If $E_i$ is the bundle appearing in $\sigma_i$ formed by $\gamma_i$ with slope between 0 and 1, then order the $\sigma_j$ by setting $\mu(E_1) < \mu(E_2) < \mu(E_3)$. For $j = 1$, $2$, or $3$, we define subgraphs $\Gamma^{\sigma}_j$ of $\Gamma'$ by letting $\Gamma^{\sigma}_j$ consist of all the vertices of $\Gamma'$ corresponding to helices obtained from $\sigma_j$ by a sequence of admissible mutations, and edges corresponding to those mutations. 

For each helix $\sigma_j$ corresponding to the vertex $v_j$, we consider the following sequences of mutations. If $\sigma \mapsto \sigma_j$ is a right mutation, we set $\{R1^k(\sigma_j)\}_{k \geq 0}$ to be the right branch off of $\sigma_j$, and if $\sigma \mapsto \sigma_j$ is a left mutation, we set $\{R2^k(\sigma_j)\}_{k \geq 0}$ to be the right branch. In the right mutation case, we set $\{R0^k(\sigma_j)\}_{k \geq 0}$ to be the left branch, and in the left mutation case, $\{L1^k(\sigma_j)\}_{k \geq 0}$ to be the left branch.

\begin{lemma} \label{mini}
	Suppose that $\gamma: \sigma \mapsto \tau$ is a right mutation producing $F$ and $F_1, F_2$, and $F_3$ are its admissible mutation bundles. Then 
		\[\mu(F_1) < \mu(R1(F_1)) < \mu(L1(F_2)) < \mu(F_2).\]
\end{lemma}

\begin{proof}
	The first and last inequalities follow from the definitions of the mutations $R1$ and $L1$. For the middle one, set $(E,F,G,H)$ to be a foundation for $\tau$. Then the admissible mutations have foundations $(E, F_1, F, G)$, $(F, F_2, G, H)$, and $(E, F_3, F, H)$ respectively. The first two of these are commuting, with a non-admissible mutation $(E, F_1, F_2, G)$ of both. Then we have another helix with foundation $(F_1, R1(F_1), L1(F_2), F_2)$, and the middle inequality of the lemma follows from Lemma \ref{mutations}.
\end{proof}

\begin{lemma} \label{tree-lemma}
	Suppose that $i < j$. Then the right branch off of $\sigma_i$ is disjoint from the left branch off of $\sigma_j$. In particular, for each vertex $v_{i,k}$ of the right branch off of $\sigma_i$ and each vertex $v'_{j, \ell}$ of the left branch of $\sigma_j$, we have $\mu(v_{i,k}) < \mu(v'_{j, \ell})$.
\end{lemma}

\begin{proof}
	The second statement obviously implies the first. We can reduce the proof of the lemma to two cases, as follows. We will consider the case where $E$ is a right mutation, the left mutation case being entirely similar. Note that each of the descendants of $E_2$ have slope less than each of the descendants of $E_3$. This is because the (admissible) descendants of $E_2$ have slope less than the slope of $E$, and the descendants of $E_3$ have slope greater than the slope of $E$; these statements both follow from Lemma \ref{mutations}. Thus it is enough to separate the branches off of $E_i = E_1$ and $E_j = E_2$, and for that, it is enough to consider the rightmost branch off of $E_1$ and the leftmost branch off of $E_2$, ordered by slope. Note that the rightmost branch off of $E_1$ is obtained through successive $R1$ mutations, and the leftmost branch off of $E_2$ is obtained through successive $L1$ mutations.
			
	Consider the limit $\overline{\mu}$ of the slopes $\mu(v_{1,k})$ as $k$ goes to infinity (on the right branch), and the limit $\overline{\mu}'$ of the slopes $\mu(v_{2, \ell})$ as $\ell$ goes to infinity (on the left branch). Both limits exist because the sequences are monotone and bounded below by 0 and above by 1. We will show that $\overline{\mu} < \overline{\mu}'$, which is sufficient to prove the lemma.
	
	We can compute the limits $\overline{\mu}$ and $\overline{\mu}'$ explicitly, as follows. For each $k$, set $R_k$ to be the unique bundle with slope between $0$ and $1$ by the mutation $R1^k$ in the right branch off of $\sigma_i$. Note that there are canonical isomorphisms
		\[\Hom(R_k, R_{k+1}) \simeq \Hom(R_{k+1}, R_{k+1})^{\vee}\]
	so in particular their dimensions agree. Set $h_R$ for this value on the right branch off of $\sigma_i$ and $h_L$ for this value along the left branch off of $\sigma_j$. Then along the right branch, we have
		\[r(R_{k+1}) = h_R r(R_k) - r(R_{k-1}), \quad d(R_{k+1}) = h_R d(R_k) - d(R_{k-1}).\]
	The rank $r(R_n)$ and degree $d(R_n)$ can be written explicitly by solving these generalized Fibonacci relations.
	
	To do so, let us focus first on the right branch $\{R_k\}$. Set $r_i = r(R_i)$ for each $i$ and $d_i = d(R_i)$. Then we set 
		\[R(z) = \sum_{n \geq 0} r_nz^n, \quad D(z) = \sum_{n \geq 0} d_nz^n\]
	to be the associated generating functions. We have
	
	\begin{align*}
		R(z) &= r_0 + r_1z + \sum_{n \geq 0} (h_R r(R_{n+1}) - r(R_n))z^{n+2} \\
		&= r_0 + r_1z + h_R z \sum_{n \geq 0} r(R_{n+1}) z^{n+1} - z^2\sum_{n \geq 0} r(R_n)z^n \\
		&= r_0 + r_1z + h_R z (R(z) - r_0) - z^2 R(z)
	\end{align*}

Therefore
	\[R(z) = \frac{r_0 + (r_1-h_R r_0)z}{z^2 - h_R z + 1}.\]
Let $\varphi_R$ and $\psi_R$ denote the roots of the polynomial $z^2 - h_R z + 1$. Both $\varphi_R$ and $\psi_R$ are positive, and without loss of generality we have $0 < \varphi_R < 1 < \psi_R$. 

We solve the partial fraction decomposition
	\[R(z) = \frac{A_R}{z - \varphi_R} + \frac{B_R}{z - \psi_R}\]
with
	\[A_R = \frac{r_0 + (r_1-h_R r_0)\varphi_R}{\varphi_R - \psi_R}, \quad B_R = \frac{r_0 + (r_1 - h_R r_0)\psi_R}{\psi_R - \varphi_R}. \]
	
Entirely similarly, we have
	\[D(z) = \frac{C_R}{z - \varphi_R} + \frac{D_R}{z - \psi_R}\]
with
	\[C_R = \frac{d_0 + (d_1 - h_R d_0)\varphi_R}{\varphi_R- \psi_R}, \quad D_R = \frac{d_0 + (d_1 - h_R d_0) \psi_R}{\psi_R - \varphi_R}.\]

Now we can solve, using the fact that $\varphi_R \psi_R= 1$,
	\begin{align*}
		R(z) &=  \frac{A_R}{z - \varphi_R} + \frac{B_R}{z - \psi_R} \\
		&= A_R \frac{1}{z - \varphi_R} + B_R \frac{1}{z - \psi_R} \\
		&= -A_R \frac{1/\varphi_R}{1-\psi_R z}  - B_R \frac{1/\psi_R}{1 - \varphi_Rz} \\
		&= -\frac{A_R}{\varphi_R} \sum_{n \geq 0} (\psi_R z)^n - \frac{B_R}{\psi_R} \sum_{n \geq 0} (\varphi_R z)^n. 
	\end{align*}

Thus we have
	\[r_n = -\frac{A_R}{\varphi_R} \psi_R^n - \frac{B_R}{\psi_R} \varphi_R^n\]
and similarly
	\[d_n = -\frac{C_R}{\varphi_R} \psi_R^n - \frac{D_R}{\psi_R} \varphi_R^n.\]

In particular, we have
	\[\mu(R_n) = \frac{d_n}{r_n} = \frac{\frac{C_R}{\varphi_R} \psi_R^n + \frac{D_R}{\psi_R} \varphi_R^n}{\frac{A_R}{\varphi_R} \psi_R^n + \frac{B_R}{\psi_R} \varphi_R^n}.\]
Multiplying top and bottom by $\varphi^n$ and taking the limit as $n \rightarrow \infty$, we obtain
	\[\overline{\mu} = \lim_{n \rightarrow \infty} \frac{\frac{C_R}{\varphi_R} + \frac{D_R}{\psi_R} \varphi_R^{2n}}{\frac{A_R}{\varphi_R} + \frac{B_R}{\psi_R} \varphi_R^{2n}} = \frac{C_R}{A_R}\]
as $0 < \varphi_R < 1$ so $\varphi_R^n \rightarrow 0$.

Explicitly, this is
	\[\overline{\mu} = \frac{d_0 + (d_1 - h_R d_0) \varphi_R}{r_0 + (r_1 - h_R r_0) \varphi_R} =: \frac{\overline{d}}{\overline{r}}.\]

Recall that we want to show that $\overline{\mu} < \overline{\mu}'$. The same process as above gives us
	\[\overline{\mu}' = \frac{d'_0 + (d'_1 - h_L d'_0) \varphi_L}{r'_0 + (r'_1 - h_L r'_0) \varphi_L} =: \frac{\overline{d}'}{\overline{r}'}.\]

It is sufficient to show that $\overline{d}'\overline{r} - \overline{d}\overline{r}' > 0$. We have
	\begin{align} \label{brrr}
		\overline{d}' \overline{r} - \overline{d} \overline{r}' &= (d'_0 + (d'_1 - h_L d'_0) \varphi_L)(r_0 + (r_1 - h_R r_0) \varphi_R) \nonumber \\
		&\qquad - (d_0 + (d_1 - h_R d_0) \varphi_R)(r'_0 + (r'_1 - h_L r'_0) \varphi_L) \nonumber \\
		&= ((d'_1 - h_L d'_0)(r_1 - h_R r_0) - (d_1 - h_R d_0)(r'_1 - h_L r'_0)) \varphi_R \varphi_L \nonumber \\ 
		&\quad + (d_0'(r_1 - h_R r_0) - (d_1 - h_R d_0)r'_0) \varphi_R + ((d'_1 - h_Ld'_0)r_0 - d_0(r'_1 - h_L r'_0)) \varphi_L \\
		& \qquad + (d'_0r_0 - d_0r'_0). \nonumber
	\end{align}

We will show that each coefficient on $\varphi_L \varphi_R, \varphi_R, \varphi_L, 1 > 0$ is positive. Note that in the mutation forming $R_1$ from $R_0$, we have an exact sequence
	\[0 \rightarrow Z \rightarrow R_0^{\oplus h_R} \rightarrow R_1 \rightarrow 0\]
so that $d(Z) = h_R d_0 - d_1$ and $r(Z) = h_R r_0 - r_1$. Similarly, we have an exact sequence
	\[0 \rightarrow L_1 \rightarrow L_0^{\oplus h_L} \rightarrow Z' \rightarrow 0\]
with $d(Z') = h_L d_0' - d_1'$ and $r(Z) = h_L r_0' - r_1'$. We have $\mu(Z) < \mu(E_0) < \mu(E_1) < \mu(E_1') < \mu(E_0')$ by Lemmas \ref{mutations} and \ref{mini}.

Then the constant term of (\ref{brrr}) is positive since $\mu(E_0') > \mu(E_0)$. The coefficient on $\varphi_R$ is positive since $\mu(E_0') > \mu(Z)$, the coefficient on $\varphi_L$ is positive since $\mu(E_0) < \mu(Z')$, and the coefficient on $\varphi_L \varphi_R$ is positive since $\mu(Z) < \mu(Z')$. Thus $\overline{\mu} < \overline{\mu}'$, and the lemma is proved.
\end{proof}

The following proposition, which shows that the inductive structure of $\Gamma_{\PP^3}'$ is especially simple, is a corollary of Lemma \ref{tree-lemma}. It will be fundamental to establishing our main theorems.

\begin{proposition} \label{gamma-tree}
	The graph $\Gamma'_{\PP^3}$ is a connected tree; every vertex except for that corresponding to $\sigma_0$ has degree 4, and $\sigma_0$ has degree 2. Every constructive exceptional bundle $E$ of slope $0 < \mu(E) < 1$ occurs in a helix in $\Gamma_{\PP^3}'$.
\end{proposition}

\begin{proof}
	The fact that $\Gamma_{\PP^3}'$ is a tree follows from Lemma \ref{tree-lemma}. Specifically, the branches off of a single vertex are disjoint from the lemma. The proof of the lemma shows that to check disjointness of two branches, all that is needed is that the first two leaves in the first branch are bounded by the first two leaves in the other branch. This is clearly the case for branches off of distinct vertices as well. 
	
	The proof of Lemma \ref{tree-lemma} also shows that the branches of this tree are ordered by the slope of the bundle produced. In particular, it follows that no admissible mutation can produce a bundle of slope between 0 and 1 from a bundle with slope less than 0 or larger than 1. It follows from this and Proposition \ref{adm-replacement} that each constructive exceptional bundle of slope between 0 and 1 appears in $\Gamma_{\PP^3}'$, and connectedness also follows from this.
\end{proof}

\section{Classifying constructive exceptional bundles} \label{class-section}

In this section we construct a bijection $\epsilon_{\PP^3}$ between the 3-adic rationals and the Chern characters of constructive exceptional bundles on $\PP^3$, analogous to Dr\'ezet and Le Potier's construction of $\epsilon_{\PP^2}$ on $\PP^2$. We begin with some notational set-up.

\begin{notation}
The rational Grothendieck group $K(\PP^3)_{\mathbb{Q}} = K(D^b(\PP^3))_{\mathbb{Q}}$ of $\PP^3$ is isomorphic to the rational Chow ring $A^*(\PP^3) \otimes \QQ$ (or singular cohomology $H^*(\PP^3, \QQ)$) via the Chern character map 
	\[\ch: K(\PP^3) \otimes \QQ \rightarrow A^*(\PP^3) \otimes \QQ \simeq \QQ^4, \quad v \mapsto (\ch_0(v).H^3, \ch_1(v).H^2, \ch_2(v).H, \ch_3(v))\]
where $H$ is the hyperplane class. The Chern characters $\ch_i$ can be regarded as $\QQ$-valued functions $\ch_i = \ch_i.H^{3-i}$ on $K(\PP^3)$. These form an integral basis, and we use them as coordinates.

Hirzebruch-Riemann-Roch (which we will often shorten to simply Riemann-Roch) expresses the bilinear form $\chi(-,-)$ on $K(\PP^3)$ in terms of these coordinates as follows:
	\[\chi(\OO, E) = \chi(E) = \ch_0(E) + \frac{11}{6}\ch_1(E) + 2\ch_2(E) + \ch_3(E)\]
(where the coefficients can be read off as the Todd class of $\PP^3$ backwards), and
	\begin{align*}
		\chi(E,F) = \chi(F \otimes E^{\vee}) &= \ch_0(F)\ch_0(E) \\
		 &+ \frac{11}{6}(\ch_1(F)\ch_0(E) - \ch_1(E)\ch_0(F)) \\
		&+ 2(\ch_2(F)\ch_0(E)-\ch_1(F)\ch_1(E)+\ch_2(E)\ch_0(F)) \\
		&+ \ch_3(F)\ch_0(E) - \ch_2(F)\ch_1(E) + \ch_2(E)\ch_1(F) - \ch_3(E)\ch_0(F).
	\end{align*}
It follows that for a nonzero class $v \in K(\PP^3)$, the conditions $\chi(-, v) = 0$ and $\chi(v,-) = 0$ are linear conditions in the Chern character basis for $K(\PP^3)$.

In particular, we will also consider the affine space patch on the projectivization $\PP(K(\PP^3)_{\QQ})$ defined by setting $\ch_0 \neq 0$. It is a three-dimensional rational vector space with basis $(\frac{\ch_1}{\ch_0}, \frac{\ch_2}{\ch_0}, \frac{\ch_3}{\ch_0})$. Any sheaf $E$ on $\PP^3$ with nonzero rank defines a unique point in this space.
\end{notation}

The following definition encodes the data of cohomological orthogonality among exceptional collections.

\begin{definition}
Let $E$, $G$, and $H$ be sheaves on $\PP^3$ (or classes in $K(\PP^3)$). Then, by Hirzebruch-Riemann-Roch, the conditions
	\[\chi(-, E) = \chi(G, -) = \chi(H, -) = 0\]
are linear conditions on $K(\PP^3) \simeq \Z^4$. When $E$, $G$, and $H$ are independent in $K(\PP^3)$, they determine a 1-dimensional subspace $L \subseteq \Z^4$; if $L$ is not contained in the locus $(\ch_0 = 0) \subseteq K(\PP^3)$, then it determines a unique point $\ell$ in the $\frac{\ch_i}{\ch_0}$-space defined above, and we define
	\[\perpp(E, G, H) := \ell \in \left\{ \frac{\ch_1}{\ch_0}, \frac{\ch_2}{\ch_0}, \frac{\ch_3}{\ch_0} \right\}.\] 
Explicit formulas for $\text{perp}(E,G,H)$ in terms of the Chern characters of $E$, $G$, and $H$ can be easily given, but are rather complicated and we will not use them, so we do not list them. 

When $(Z, E, G, H)$ are the foundation of a constructive helix on $\PP^3$, they form a full exceptional collection for $D^b(\PP^3)$. The right mutation bundle $F = R_E Z$ fits into a new exceptional collection $(E, F, G, H)$, and $F$ satisfies the linear conditions above. Similar conditions are satisfied by left mutations. The point $\ell$ above corresponding to $F$ is thus well-defined, and since $\ch_1$ and $\ch_0$ are coprime for exceptional bundles, determines the full Chern character of $F$ in terms of the Chern characters of $E$, $G$, and $H$. 
\end{definition}

\begin{example}
We compute perp for the exceptional collection $(\OO, T(-1), \OO(1), \OO(2))$, to recompute $\ch(T(-1))$. This can easily be done by using the defining exact sequence, but is more convenient to do with the perp operator in Theorem \ref{theorem}. 

The Chern character $\ch(T(-1))$ satisfies the orthogonality relations
	\[\chi(T(-1), \OO) = \chi(\OO(1), T(-1)) = \chi(\OO(2), T(-1)) = 0.\]
Set $\ch_i = \ch_i(T(-1))$. Via Riemann-Roch, these equations can be written
	\begin{align*}
		\chi(T(-1), \OO) &= \ch_0 - \frac{11}{6}\ch_1 + 2 \ch_2 - \ch_3 = 0 \\
		\chi(\OO(1), T(-1)) &= \ch_0 + \frac{11}{6}(\ch_1 - \ch_0) + 2(\ch_2 - \ch_1 + \frac{1}{2}\ch_0) \\ &\qquad + \ch_3 - \ch_2 + \frac{1}{2}\ch_1 - \frac{1}{6}\ch_0 = 0 \\
		\chi(\OO(2), T(-1)) &= \ch_0 + \frac{11}{6}(\ch_1 - 2\ch_0) + 2(\ch_2 - 2\ch_1 + 2\ch_0) \\ &\qquad + \ch_3 - 2\ch_2 + 2\ch_1 - \frac{4}{3}\ch_0 = 0
	\end{align*}
Dividing by $\ch_0 \neq 0$, we obtain three linear equations in the three variables $\frac{\ch_1}{\ch_0}$, $\frac{\ch_2}{\ch_0}$, and $\frac{\ch_3}{\ch_0}$, which is easily solved by hand or by computer. Then Riemann-Roch implies that $\ch_1$ and $\ch_0$ are coprime, so one can determine $\ch_0$ and the rest of the Chern character from these. One can then easily check that we get
	\[\ch(T(-1)) = \text{perp}(\OO, \OO(1), \OO(2)) = \left(3, 1, \frac{-1}{2}, \frac{1}{6}\right).\]
\end{example}

\begin{theorem} \label{theorem}
Let $F$ be a constructive exceptional bundle on $\PP^3$. Then there is a well-defined way to choose three other constructive exceptional bundles $E_1(F), E_2(F), E_3(F)$ on $\PP^3$ such that $(E_1(F), F, E_2(F), E_3(F))$ form a full exceptional collection for $\PP^3$. This association produces a bijection
	\[\epsilon_{\PP^3}: \{\text{3-adic rationals}\} \rightarrow \left\{\begin{array}{c} \text{Chern characters of constructive} \\ \text{exceptional bundles on } \PP^3 \end{array} \right\}\]
defined inductively as
	\[\epsilon_{\PP^3}(n) = \mathrm{ch}(\mathcal{O}(n)) = \left(1, n, \frac{n^2}{2}, \frac{n^3}{6}\right) \quad n \in \Z,\]
and for $q \geq 0$,
	\begin{align*}
		\epsilon_{\PP^3}\left( \frac{3p+1}{3^{q+1}} \right) &= \mathrm{perp}\left( \epsilon_{\PP^3}\left( \frac{p}{3^q} \right), E_2\left(\epsilon_{\PP^3}\left( \frac{p}{3^q} \right)\right), E_3\left( \epsilon_{\PP^3}\left( \frac{p}{3^q} \right)\right) \right) \\
		\epsilon_{\PP^3}\left( \frac{3p+2}{3^{q+1}} \right) &= \mathrm{perp}\left( E_1\left( \epsilon_{\PP^3}\left( \frac{p+1}{3^q} \right) \right), \epsilon_{\PP^3}\left( \frac{p+1}{3^q} \right), E_3\left( \epsilon_{\PP^3}\left( \frac{p+1}{3^q} \right) \right) \right).
	\end{align*}
\end{theorem}

As in the statement of the theorem, we will sometimes abuse notation and denote by $\epsilon_{\PP^3}(t/3^q)$ both the Chern character and the constructive exceptional bundle with this Chern character. There is no ambiguity here because the definition of $\epsilon_{\PP^3}$ will determine a set of exact sequences which determine the both the Chern character and the sheaf.

\begin{definition}
When $\epsilon(\frac{3p+t}{3^q}) = \ch(E)$ and $\frac{3p+t}{3^q}$ is reduced, we call $q$ the \textit{order} of $E$. The order of a line bundle $\OO(a)$ is 0, the tangent and cotangent bundles $T, T^{\vee}$ are of order 1.
\end{definition}

The description of the graphs $\Gamma_{\PP^3}$ and $\Gamma_{\PP^3}'$ in Section \ref{graphs-section}, especially Proposition \ref{adm-replacement}, produce the required helix-theoretic information to prove the theorem.

\begin{proof}[Proof of Theorem \ref{theorem}.]
We reduce by twisting by line bundles to define the map on rationals between 0 and 1, which will take values in the Chern characters of constructive exceptional bundles $E$ whose slope $0 < \mu(E) < 1$. The proof will be inductive on the ``order'' $q$ of the rational numbers $t/3^q$. The base cases are when $q = 0$ or $1$; when $q = 0$ we have $\frac{t}{3^q} = t = 0$ or 1, and we assign these respectively to $\OO$ and $\OO(1)$. Both lie in the standard helix $\sigma_0$, but in this base case it won't be necessary to choose a foundation which contains them. When $q=1$, the construction below applies but it will be necessary to pick a foundation to start with. We have two rationals, $\frac{1}{3}$ and $\frac{2}{3}$, to assign, and we set
	\[\epsilon\left( \frac{1}{3} \right) = \mathrm{ch}(T(-1)) = \mathrm{perp}( \epsilon(0), \epsilon(1), \epsilon(2) ) = perp(\OO, \OO(1), \OO(2))\]
and
	\[\epsilon\left( \frac{2}{3} \right) = \mathrm{ch}(T^{\vee}(2)) = \mathrm{perp}(\epsilon(0), \epsilon(2), \epsilon(4)) = perp(\OO, \OO(2), \OO(3)).\]
We note that we obtain $T(-1)$ from $\sigma_0$ as the right mutation $R_{\OO}\OO(-1)$, and $T^{\vee}(2)$ as the left mutation $L_{\OO(1)}\OO(2)$.

In higher orders, we first define the value of $\epsilon$ on the rationals that will be associated to the right mutations, which are of the form $\frac{3p+1}{3^{q+1}}$. We distinguish the largest 3-adic rational smaller than $\frac{3p+1}{3^{q+1}}$ of order smaller than $q+1$, namely $\frac{p}{3^q}$. When $\frac{p}{3^q}$ is reduced, we may by induction write $\epsilon(\frac{p}{3^q}) = ch(F)$ and include $F$ into to a constructive helix $\sigma$, with a choice of foundation $(E,F,G,H)$. We have $p = 1$ or $2$ mod 3. In the first case, we apply the mutation $R1$ to $\sigma$ to obtain a helix with foundation $(F, R_FE, G, H)$. We set
	\[\epsilon\left( \frac{3p+1}{3^{q+1}} \right) = \text{ch}(R_FE) = \mathrm{perp}(E, G, H).\]
If $\frac{p}{3^q}$ is not reduced, write $\frac{p}{3^q} = \frac{p'}{3^{q'}}$ in reduced form, and set $d = q-q'$. By induction, we can write $\epsilon(\frac{p'}{3^{q'}}) = \ch(F)$ and include $F$ into a helix $\sigma$ with foundation $(E,F,G,H)$. When $p' = 1$ mod 3, we apply the mutation $R0$ $d$ times, and when $p = 2$ mod 3, we apply the mutation $R2$ once and then the mutation $R0$ $d-1$ times. Writing the foundation obtained from these mutations as $(E',F',G',\text{H}_{\PP^3}')$, with $F'$ being the bundle obtained in the final mutation, we set
	\[\epsilon\left( \frac{3p+1}{3^{q+1}} \right) = \text{ch}(F) = \mathrm{perp}(E', G', H').\]
Now, we assign values to rationals of the form $\frac{3p+2}{3^{q+1}}$, which will be the left mutations. We distinguish the smallest 3-adic rational larger than $\frac{3p+2}{3^{q+1}}$ of order smaller than $q+1$, namely $\frac{p+1}{3^q}$. If it is reduced, we write $\epsilon(\frac{p+1}{3^{q+1}}) = ch(G)$ and include $G$ into a foundation $(E,F,G,H)$ for a constructive helix $\sigma$ by induction. If $p+1 = 1$ mod 3, we apply the mutation $L0$ once; if $p+1 = 2$ mod 3, we apply the mutation $L1$ once. We then set
	\[\epsilon\left( \frac{3p+2}{3^{q+1}} \right) = \text{ch}(L_GH) = \mathrm{perp}(F, G, E(4)).\]
When $\frac{p+1}{3^q}$ is not reduced, we write $\frac{p+1}{3^q} = \frac{p''}{3^{q''}}$ in reduced form. Set $d = q - q''$. If $p'' = 1$ mod 3, then we perform we apply one $L1$ mutation and $d-1$ $L2$ mutations; if $p'' = 2$ mod 3, we apply the mutation $L2$ $d$ times. Writing $(E',F',G',H')$ for the foundation obtained from $(E,F,G,H)$ via these mutations with $G'$ the bundle obtained from the final mutation, we set
	\[\epsilon\left( \frac{3p+2}{3^{q+1}} \right) = \text{ch}(G') = \mathrm{perp}(F', H', E'(4)).\]
This defines $\epsilon$ entirely. We now show that it is a bijection.

First, we show that $\epsilon$ is injective. We apply Proposition \ref{adm-replacement}. The admissible mutations arising from the mutations of $\sigma_0$ are the mutations $R0, L0, R1, L1, R2, L2$, which are the mutations applied in the above construction. Since no inverses are applied, it follows that each such description describes a unique path on the graph $\Gamma'$. Since $\Gamma'$ is a tree, two rationals are sent to the same Chern character if and only if the mutations described above define the same path on $\Gamma'$, which produce the same exceptional bundle.

For surjectivity, we apply the classification statement in Proposition \ref{adm-replacement}, namely that every constructive exceptional bundle of slope $0 < \mu < 1$ appears in $\Gamma'$. For a constructive exceptional bundle $E$ of slope $0 < \mu(E) < 1$, we choose a sequence of admissible mutations, starting with $\sigma_0$, which terminates in a mutation producing $E$, as in Proposition \ref{adm-replacement}. Each of these mutations, by induction, has a preferred foundation which contains a unique new bundle with slope between 0 and 1. Let $\sigma' \mapsto \sigma$ be the last mutation in this sequence, so that $E$ is contained in $\sigma$, and let $F$ denote the new bundle in $\sigma'$ with slope between 0 and 1. If $F$ is a right mutation, we set $\epsilon_{\PP^3}((3p+1)/3^{q+1}) = \text{ch}(F)$, and if $F$ is a left mutation we set $\epsilon_{\PP^3}((3p+2)/3^{q+1}) = \text{ch}(F)$, and choose a preferred foundation as in the theorem statement in either case. 

We will only explicitly describe which triadic rational to assign to $E$ in the case where $F$ is a right mutation, the left mutation case being entirely similar. The rational numbers of order $q+2$ which correspond to the mutations of $\sigma$ are $(9p + 2)/3^{q+2}$, $(9p+4)/3^{q+2}$, and $(9p+5)/3^{q+2}$; we assign these to $R0$, $L0$, and $L1$, respectively. This concludes the proof.
\end{proof}

\subsection{Explicit computations} \label{explicit-section}
In this subsection we list the numerical and cohomological data for the first few exceptional bundles with slope between 0 and 1. Each appears in a foundation obtained from the standard helix with foundation $(\OO(-1), \OO, \OO(1), \OO(2))$ via admissible mutations. Each such mutation produces a standard resolution by lower order exceptional bundles; each also fits into a helix triangle which produces the same bundle as the opposite mutation. These mutations/resolutions will be called \textit{standard}, and we list them below; see Figure \ref{fig:table}. 

\begin{figure}[!] 
\adjustbox{scale = .8,center}{
	\renewcommand{\arraystretch}{1.3}
		\begin{tabular}{|c|c|c|c|c|c|c|} \hline
		$E_{\mu}$ & order & \begin{tabular}{c} standard \\ foundation  \end{tabular} & \begin{tabular}{c} standard \\ resolutions  \end{tabular} & Chern character & $\epsilon(\frac{p}{3^q})$ & \begin{tabular}{c} nonzero \\ cohomology  \end{tabular} \\  \hline 
		$E_n = \OO(n)$ & 0 & $(\OO(-1), \OO, \OO(1), \OO(2))$ & \textit{n/a} & $(1,n,\frac{n^2}{2}, \frac{n^3}{6})$ & $\epsilon(n)$ & $h^0 = {n+3 \choose 3}$ \\ \hline
		$E_{1/3} = T(-1)$ & 1 & $(\OO, T(-1), \OO(1), \OO(2))$ & $\begin{array}{c} \OO(-1) \rightarrow \OO^4 \rightarrow T(-1) \\ T(-1) \rightarrow \OO(1)^6 \rightarrow T^{\vee}(3) \end{array}$ & $(3,1,\frac{-1}{2}, \frac{1}{6})$ & $\epsilon(\frac{1}{3})$ & $h^0 = 4$ \\ \hline
		
		$E_{2/9} = R0(T(-1))$ & 2 & $(\OO, E_{2/9}, T(-1), \OO(1))$ & $\begin{array}{c} \OO(-2) \rightarrow \OO^{10} \rightarrow E_{2/9} \\ E_{2/9} \rightarrow T(-1)^4 \rightarrow T^{\vee}(2) \end{array}$ & $(9,2,-2,\frac{4}{3})$ & $\epsilon(\frac{1}{9})$ & $h^0 = 10$  \\ \hline
		
		$E_{5/17} = L1(T(-1))$ & 2 & $(\OO(-2), \OO, E_{5/17}, T(-1))$ & $\begin{array}{c} E_{5/17} \rightarrow T(-1)^6 \rightarrow \OO(1) \\ T(-3) \rightarrow \OO^{20} \rightarrow E_{5/17} \end{array}$ & $(17,5,\frac{-7}{2}, \frac{5}{6})$ & $\epsilon(\frac{2}{9})$ & $h^0 = 20$  \\ \hline
		
		$E_{4/11} = R1(T(-1))$ & 2 & $(T(-1), E_{4/11}, \OO(1), \OO(2))$ & $\begin{array}{c} \OO \rightarrow T(-1)^4 \rightarrow E_{4/11} \\ E_{4/11} \rightarrow \OO(1)^{20} \rightarrow E_{7/9}(1) \end{array}$ & $(11,4,-2, \frac{2}{3})$ & $\epsilon(\frac{4}{9})$ & $h^0 = 15$ \\ \hline
		
		$E_{3/19} = R0(E_{2/9})$ & 3 & $(\OO, E_{3/19}, E_{2/9}, T(-1))$ & $\begin{array}{c} \OO(-3) \rightarrow \OO^{20} \rightarrow E_{3/19} \\ E_{3/19} \rightarrow E_{2/9}^4 \rightarrow E_{5/17} \end{array}$ & $(19, 3, \frac{-9}{2}, \frac{9}{2})$ & $\epsilon(\frac{1}{27})$ & $h^0 = 20$ \\ \hline
		
		$E_{7/33} = L1(E_{2/9})$ & 3 & $(\OO, E_{7/33},E_{2/9},\OO(1))$ & $\begin{array}{c} E_{7/33} \rightarrow E_{2/9}^4 \rightarrow T(-1) \\ T^{\vee}(3) \rightarrow \OO(4)^{36} \rightarrow E_{7/33}  \end{array}$ & $(33,7,\frac{-15}{2}, \frac{31}{6})$ & $\epsilon(\frac{2}{27})$ & $h^0 = 36$ \\ \hline
		
		$E_{20/89} = R1(E_{2/9})$ & 3 & $(E_{2/9}, E_{20/89}, T(-1),\OO(1))$ & $\begin{array}{c} \OO \rightarrow E_{2/9}^{10} \rightarrow E_{20/89} \\ E_{20/89} \rightarrow T(-1)^{36} \rightarrow E_{3/19}^{\vee} \end{array}$ & $(89,20,-20,\frac{40}{3})$ & $\epsilon(\frac{4}{27})$ & $h^0 = 99$ \\ \hline
		
		$E_{29/99} = L1(E_{5/17})$ & 3 & $(\OO(-2), \OO, E_{29/99},E_{5/17})$ & $\begin{array}{c} E_{29/99} \rightarrow E_{5/17}^6 \rightarrow T(-1) \\ E_{5/17}(2) \rightarrow \OO(4)^{116} \rightarrow E_{29/99} \end{array}$ & $(99,29,\frac{-41}{2}, \frac{29}{6})$ & $\epsilon(\frac{5}{27})$ & $h^0 = 116$ \\ \hline
		
		$E_{100/339} = R2(E_{5/17})$ & 3 & $(E_{5/17},E_{100/339},T(-1),\OO(2))$ & $\begin{array}{c} \OO \rightarrow E_{5/17}^{20} \rightarrow E_{100/339} \\ E_{100/339} \rightarrow T(-1)^{116} \rightarrow E_{7/9}(1) \end{array}$ & $(339,100,-70,\frac{50}{3})$ & $\epsilon(\frac{7}{27})$ & $h^0 = 399$ \\ \hline
		
		$E_{18/59} = L2(E_{5/17})$ & 3 & $(\OO, E_{5/17}, E_{18/59}, T(-1))$ & $\begin{array}{c} E_{18/59} \rightarrow T(-1)^{20} \rightarrow \OO(2) \\ E_{2/9} \rightarrow E_{5/17}^4 \rightarrow E_{18/59} \end{array}$ & $(59,18,8,2)$ & $\epsilon(\frac{8}{27})$ & $h^0 = 70$ \\ \hline
		
		$E_{38/107} = R0(E_{4/11})$ & 3 & $(T(-1), E_{38/107},E_{4/11}, \OO(1))$ & $\begin{array}{c} \OO(-2) \rightarrow T(-1)^{36} \rightarrow E_{38/107} \\ E_{38/107} \rightarrow E_{4/11}^{10} \rightarrow T^{\vee}(2) \end{array}$ & $(107,38,-20,\frac{22}{3})$ & $\epsilon(\frac{10}{27})$ & $h^0 = 144$ \\ \hline
		
		$E_{79/219} = L1(E_{4/11})$ & 3 & $(\OO(-2), T(-1), E_{79/219}, E_{4/11})$ & $\begin{array}{c} E_{79/219} \rightarrow E_{4/11}^{20} \rightarrow \OO(1) \\ T(-3) \rightarrow T(-1)^{74} \rightarrow E_{79/219} \end{array}$ & $(219,79,\frac{-81}{2}, \frac{79}{6})$ & $\epsilon(\frac{11}{27})$ & $h^0 = 296$ \\ \hline
		
		$E_{15/41} = R1(E_{4/11})$ & 3 & $(E_{4/11}, E_{15/41}, \OO(1), \OO(2))$ & $\begin{array}{c} T(-1) \rightarrow E_{4/11}^4 \rightarrow E_{15/41} \\ E_{15/41} \rightarrow \OO(1)^{74} \rightarrow E_{7/33}^{\vee}(2) \end{array}$ & $(41,15,\frac{-15}{2},\frac{15}{6})$ & $\epsilon(\frac{13}{27})$ & $h^0 = 56$ \\ \hline
	\end{tabular}
}

\caption{The first few exceptional bundles, with numerical and cohomological invariants. Note that there is an involution on the set of exceptional slopes between 0 and 1 given by $\mu \mapsto 1 - \mu$, so the bundles which are descended from $T^{\vee}_{\mathbb{P}^3}(2)$ rather than $T_{\mathbb{P}^3}(-1)$ can be found from these quickly.}
\label{fig:table}
\end{figure}

These explicit computations and many others suggest the following conjecture.

\begin{conjecture}[Weak Brill-Noether for exceptional bundles] \label{wbn}
Let $E$ be a constructive exceptional bundle on $\PP^3$. Then $h^i(E) \neq 0$ for at most one $i$, and
	\[\chi(E) = \begin{cases} 
		h^0(E) & \text{if } \mu(E) \geq 0, \\
		-h^1(E) & \text{if } -4 < \mu(E) < 0 \text{ and } \chi(E) < 0, \\
		h^2(E) & \text{if } -4 < \mu(E) < 0 \text{ and } \chi(E) \geq 0, \\
		-h^3(E) & \text{if } \mu(E) < -4.
	\end{cases}\]
\end{conjecture}

Conjecture \ref{wbn} has been verified in a few cases where the exceptional bundles are sufficiently positive \cite[Corollary 3.4]{CHS}, but in general relies on proving that certain maps among cohomology groups for exceptional bundles have maximal rank, which is very subtle.

\section{Global generation} \label{gg-section}

Recall that given a mutation $\gamma: \sigma \mapsto \tau$, there are three mutations of $\tau$ that are $\gamma$-admissible; see Example \ref{adm-ex}. To prove global generation, we will only need to consider these mutations; see Figure \ref{fig:adm-graph} for a drawing of the these as a graph.

\begin{theorem} \label{gg}
If $E$ is a constructive exceptional bundle of slope $\mu \geq 0$, then $E$ is globally generated.
\end{theorem}

\begin{proof}
First, recall in the proof of Theorem \ref{theorem} that to each constructive exceptional bundle $E$ we have associated a distinguished helix $\sigma = \sigma(E)$ containing $E$. We say $E$ is a right mutation if the final mutation forming $E$ in the sequence described in the classification is a right mutation, and similary for left mutations. If $E$ is a right mutation we include it into a foundation for $\sigma$ of the form $(A, E, B, C)$, and the Chern character of $E$ is associated to a rational number of the form $\frac{3p+1}{3^q}$. If $E$ is a left mutation we include it into a foundation for $\sigma$ of the form $(A, B, E, C)$, and the Chern character of $E$ is associated to a rational number of the form $\frac{3p+2}{3^q}$. In each case, $q = \text{ord}(E)$. 

To prove the theorem we induct on the order of $E$. If the order of $E$ is zero then $E$ is a line bundle, and global generation when $\mu(E) \geq 0$ is clear. Otherwise, we may write $E$ as a left or right mutation
	\[0 \rightarrow E \rightarrow C \otimes \Hom(C,F) \rightarrow F \rightarrow 0, \quad 0 \rightarrow F \rightarrow A \otimes \Hom(F,A)^{\vee} \rightarrow E \rightarrow 0\]
where $C$ (resp. $A$) has order one less than $E$. 

We address the right mutation case first. In this case, by construction we have $\ch(A) = \epsilon_{\PP^3}(\frac{p}{3^{q-1}})$. Whenever $\mu(\epsilon_{\PP^3}(x)) > 0$, we have $x > 0$. Thus we have $3p+1 > 0$, hence $p \geq 0$. So since $\epsilon_{\PP^3}$ is an increasing function, $\mu(A) \geq 0$. By induction on the order we can assume $A$ is globally generated, so $E$ is a quotient of a globally generated bundle. It follows that $E$ is globally generated. 

In the left mutation case, the first step is to reduce to the case of a right mutation using the helix relation. Recall that in this case, $\sigma$ has a foundation $(A,B,E,C)$. Writing $E = L_CF$, the helix from which $E$ is formed by mutation has foundation $(A,B,C,F)$. The helix relation takes the form $L_C = R_BR_{F(-4)}$. These two right mutations are defined via exact sequences
	\[0 \rightarrow F(-4) \rightarrow A \otimes \Hom(F(-4), A)^{\vee} \rightarrow R_AF(-4) = E' \rightarrow 0\]
and
	\[0 \rightarrow E' \rightarrow B \otimes \Hom(E',B)^{\vee} \rightarrow R_BE' = E \rightarrow 0.\]

The next step is to show that $\mu(B) \geq 0$, in which case we can again use induction on the order to deduce that $B$ is globally generated, from which it will again follow from the second exact sequence that $E$ is globally generated as well.

We first introduce some notation and terminology. Given a constructive exceptional bundle $M$ on $\PP^3$, we can include it as in Theorem \ref{theorem} into a helix $\sigma(M)$ of order $\text{ord}(M)$. We set $\text{par}_L(M)$ and $\text{par}_R(M)$, respectively, to be the ``parent'' bundles directly to the left and right of $M$ in $\sigma$, and writing $\epsilon_{\PP^3}(\eta) = \ch(M)$, we write $\text{par}_L(\eta)$ and $\text{par}_R(\eta)$ for the rationals $\epsilon_{\PP^3}^{-1}(\text{par}_L(M))$ and $\epsilon_{\PP^3}^{-1}(\text{par}_R(M))$, respectively. Write $\epsilon_{\PP^3}(\eta) = E$, for the bundle $E$ appearing in this proof. Because $E$ is a quotient of a direct sum of copies of the bundle $B = \text{par}_L(E)$, we need to show that $\mu(\text{par}_L(E)) \geq 0$. The projection $\eta \mapsto \epsilon_{\PP^3}(\eta) \mapsto \mu(\epsilon_{\PP^3}(\eta))$ is increasing by slope-stability, so it is enough to check that $\text{par}_L(\eta) \geq 0$ when $\eta > 0$.

Using Theorem \ref{theorem} we can isolate the left mutations producing exceptional bundles and prove this by induction. As in Figure \ref{fig:adm-graph}, there are only three left mutations to consider: $L0$, $L1$, and $L2$. We will describe these mutations explicitly. We set $\eta = (3p+2)/3^{q+1}$. On foundations, we have:
	\begin{align*}
		L0: (B,\underline{C},F,A(4))& \mapsto (B, L_CF, C, A(4)) \\
		L1: (B,\underline{C},F,A(4))& \mapsto (B, L_CF, C, A(4)) \\
		L2: (A, \underline{B}, C, F) &\mapsto (A, B, L_FC, C).
	\end{align*}
where underlined bundles are new. In the first two mutations, $\text{par}_L(C) = B$ and $C$ has order less than $E = L_CF$, so we conclude by induction. In the third mutation ($L2$) we argue as follows. We will show the contrapositive: it is enough to show that if the numerator of $\epsilon_{\PP^3}^{-1}(B)$ is negative, then the numerator of $\epsilon_{\PP^3}^{-1}(L_CF)$ is nonpositive. In this case, $B$ is a left mutation, and we write $\epsilon_{\PP^3}^{-1}(B) = (3p+2)/3^q$, and we have $\epsilon_{\PP^3}^{-1}(C) = (p+1)/3^q$ and $\epsilon_{\PP^3}^{-1}(L_CF) = (3(p+1)-1)/3^{q+1}$. The numerator of $\epsilon_{\PP^3}^{-1}(B)$ is negative if and only if $p < -2/3$, if and only if the numerator of $\epsilon_{\PP^3}^{-1}(L_CF)$ is $3p+2 < 0$, as required. This concludes the proof.
\end{proof}

\begin{remark}
The same statement is true for $\PP^2$; see \cite[Proposition 4.4]{CHK}. The proof is substantially simpler on $\PP^2$, for two reasons. First, it is possible to explicitly identify the ``parent'' bundles of an exceptional bundle on $\PP^2$, and the case of a left mutation is indistinguishable from the case of a right mutation.
\end{remark}

\end{document}